\documentclass{amsart}
\usepackage{amssymb}
\usepackage{tikz-cd}
\usepackage{arydshln}

\def\mycommand{\setlength{\abovedisplayskip}{3pt}%
\setlength{\belowdisplayskip}{3pt}%
\setlength{\abovedisplayshortskip}{2pt}%
\setlength{\belowdisplayshortskip}{2pt}}

\let\oldselectfont\selectfont
\def\selectfont{\oldselectfont\mycommand}

\mycommand

\newcommand*\hexbrace[2]{%
  \underset{#2}{\underbrace{\rule{#1}{0pt}}}}

\newtheorem{theo}{Theorem}

\newtheorem{prop}[theo]{Proposition}
\newtheorem{cor}[theo]{Corollary}
\newtheorem{lem}[theo]{Lemma}
\newtheorem{rem}[theo]{Remark}

\newtheoremstyle{remarkstyle}{\topsep}{\topsep}{\rm}{}{\bfseries}{.}{.5em}{}
\theoremstyle{remarkstyle}
\newtheorem{app}[theo]{Appendix}

\newcommand{\Z}{\mathbb{Z}}
\newcommand{\Q}{\mathbb{Q}}
\newcommand{\R}{\mathbb{R}}
\newcommand{\field}{\mathbb{F}}
\newcommand{\fieldsep}{\mathbb{F}_{\text{sep}}}

\newcommand{\Efield}{\mathbb{E}}
\newcommand{\so}{\mathfrak{so}}
\newcommand{\spLie}{\mathfrak{sp}}
\newcommand{\orthoform}{\Omega}
\newcommand{\orthinv}{\tau}
\newcommand{\sympform}{\Psi}
\newcommand{\sympinv}{\psi}
\newcommand{\roots}{\Phi}
\newcommand{\Galgroup}{\Gamma} 
\newcommand{\Gmult}{\mathbb{G}_m} 

\DeclareMathOperator{\G}{\bf{G}} 
\DeclareMathOperator{\Hh}{\bf{H}} 
\DeclareMathOperator{\T}{\bf{T}} 
\DeclareMathOperator{\Spin}{\bf{Spin}} 
\DeclareMathOperator{\Spinp}{Spin} 
\DeclareMathOperator{\HSpin}{\bf{HSpin}}
\DeclareMathOperator{\HSpinp}{HSpin}
\DeclareMathOperator{\SO}{\bf{SO}}

\DeclareMathOperator{\SP}{\bf{Sp}}
\DeclareMathOperator{\SPp}{Sp}

\DeclareMathOperator{\PSP}{\bf{PSp}}
\DeclareMathOperator{\PSPp}{PSp}
\DeclareMathOperator{\PSO}{\bf{PSO}}

\DeclareMathOperator{\Sym}{S} 
\DeclareMathOperator{\Mat}{M}
\DeclareMathOperator{\Hom}{Hom}
\DeclareMathOperator{\Inv}{Inv} 
\DeclareMathOperator{\Br}{Br} 
\DeclareMathOperator{\Gal}{Gal} 

\title{Rost Multipliers of Lifted Kronecker Tensor Products}
\author[C. Ruether]{Cameron Ruether}
\address{Department of Mathematics and Statistics, University of Ottawa, 150 Louis-Pasteur
Ottawa, ON K1N 6N5 Canada. cruet042@uottawa.ca}
\thanks{Partially supported by NSERC Discovery Grant RGPIN-2015-04469}

\begin{document}
\maketitle

\begin{abstract}
We extend techniques employed by Garibaldi to construct various new injections involving the half-spin group, $\HSpin$, induced by lifting the Kronecker tensor product to simply connected groups. We calculate the Rost multipliers of the maps we have constructed. Furthermore, we utilize our new map $\PSP_{2n}\times \PSP_{2m} \hookrightarrow \HSpin_{4nm}$ to describe the structure of the normalized degree three cohomological invariants of $\HSpin_{4n}$.
\end{abstract}

\section{Introduction}
Degree three cohomological invariants of semisimple linear algebraic groups as given in \cite{GMS} have been recently studied and computed   in \cite{Baek}, \cite{BR}, \cite{GQ-M}, \cite{Mer}, \cite{MNZ} and others. Rost multipliers played an important role in those computations. In this paper, we introduce new maps between split linear algebraic groups, compute their Rost multipliers, and utilize one of the new maps to describe the structure of cohomological invariants of the split half-spin group. In particular, working over a field $\field$ of characteristic different from 2, we use the Steinberg construction of Chevalley groups to explicity describe the Kronecker tensor product map between split semisimple linear algebraic group schemes lifted to the simply connected setting (this is done in section \ref{sec:liftings}). We do so by generalizing the methods of Garibaldi from \cite[\S 7]{Gar} which are outlined in section \ref{Garibaldi_Example}. Using this explicit description we are able to compute the kernel of compositions of these maps into the half-spin group and therefore produce new injections.
\renewcommand*{\thetheo}{\Alph{theo}}
\begin{theo}
When at least one of $n$ and $m$ is even there exist commutative diagrams of split linear algebraic group schemes
\[
\begin{tikzcd}
\SP_{2n}\times \SP_{2m} \arrow[twoheadrightarrow]{d} \arrow{r} & \Spin_{4nm} \arrow[twoheadrightarrow]{d}\\
\PSP_{2n}\times\PSP_{2m} \arrow[hookrightarrow]{r} & \HSpin_{4nm}
\end{tikzcd}
\begin{tikzcd}
\Spin_{2n}\times\Spin_{2m} \arrow[twoheadrightarrow]{d} \arrow{r} & \Spin_{4nm} \arrow[twoheadrightarrow]{d} \\
\PSO_{2n}\times \PSO_{2m} \arrow[hookrightarrow]{r} & \HSpin_{4nm}.
\end{tikzcd}
\]
Furthermore, for any $n$ and $m$, there exists a commutative diagram of split linear algebraic group schemes
\[
\begin{tikzcd}
\Spin_{4n}\times\Spin_{2m+1} \arrow[twoheadrightarrow]{d} \arrow{r} & \Spin_{4n(2m+1)} \arrow[twoheadrightarrow]{d} \\
\HSpin_{4n}\times \SO_{2m+1} \arrow[hookrightarrow]{r} & \HSpin_{4n(2m+1)}
\end{tikzcd}
\]
In all diagrams the top maps are the lifted Kronecker tensor product.
\end{theo}
The Rost multipliers, integers describing induced maps between quadratic invariants, of these news maps are computed in section \ref{sec:Rost}. Our main result is in section \ref{sec:cohomology}. As an application of the existence of the map $\SP_{2n}\times\SP_{2m}\to\HSpin_{4nm}$, we prove the following structure theorem about the degree three normalized cohomological invariants of the half-spin group.
\begin{theo}
Let $n\geq 2$. We can describe the degree three normalized invariants of $\HSpin_{4n}$ as
\[
\Inv^3(\HSpin_{4n},2)_{\emph{norm}} \cong \begin{cases}
\field^\times/(\field^\times)^2 & n \text{ is odd or } n=2 \\
\field^\times/(\field^\times)^2 \oplus \Z/2\Z & n\equiv 2\pmod{4} \text{ and } n\neq 2 \\
\field^\times/(\field^\times)^2 \oplus \Z/4\Z & n\equiv 0\pmod{4}.
\end{cases}
\]
\end{theo}
This generalizes a result of Bermudez and Ruozzi in \cite{BR}. Finally in section \ref{sec:invariant}, we describe an explicit non-trivial non-decomposable normalized degree three invariant of $\HSpin_{4n}$ arising as a pull back of an invariant of $\PSO_{4n}$ constructed in \cite{Mer}. When $n\equiv 2\pmod{4}$ and $n\neq 2$, this gives a complete description of the normalized invariants of $\HSpin_{4n}$ (Remark \ref{explicit invariant}). Additionally, the appendices contain concrete descriptions of the groups $\SP,\Spin$ and $\SO$ via Chevalley generators and relations which are utilized throughout the paper.

\subsection*{Acknowledgements}
I am very grateful to the referee for their useful comments which improved the presentation of this paper substantially. In particular, simplifying the proof of theorem \ref{theo:invariant splitting}, as well as suggesting the considerations of section \ref{sec:invariant}.

\renewcommand*{\thetheo}{\arabic{theo}}
\setcounter{theo}{0}
\section{Preliminaries}\label{sec:prelim}
Throughout this paper we work over a field $\field$ of characteristic different from $2$ and deal with linear algebraic groups schemes of types B, C, and D over $\field$. We do so by mainly focusing on the group of points over $\fieldsep$, the separable closure of $\field$, and obtaining the $\field$-points via Galois descent.
\subsection{Split Linear Algebraic Group Schemes}\label{sec:split groups}
In the split case, we will consider the orthogonal involution on $\Mat_n(\field)$ given by
\[
\begin{tabular}{cc}
$\begin{aligned}
\orthinv_n \colon \Mat_n(\field) &\to \Mat_n(\field) \\
A &\mapsto \orthoform_n A^T \orthoform_n
\end{aligned}$ & \text{ where } $\orthoform_n = \begin{bmatrix} & & 1 \\ &\rotatebox{45}{$\hdots$} & \\ 1 & & \end{bmatrix}$
\end{tabular}
\]
and the associated group schemes $\PSO_n(\orthinv_n), \SO_n(\orthinv_n)$, and $\Spin_n(\orthinv_n)$. For brevity, we use the unadorned $\PSO_n,\SO_n$, and $\Spin_n$ when the orthogonal involution in question is the above $\orthinv_n$. The $\fieldsep$-points of these groups are
\begin{align*}
\SO_n(\fieldsep) &= \{ A \in \Mat_n(\fieldsep) \mid A\cdot(\orthinv_n\otimes 1)(A)=I, \det(A)=1 \} \\
\PSO_n(\fieldsep) &= \SO_n(\fieldsep)/Z(\SO_n(\fieldsep)) \\
\Spin_n(\fieldsep) &= \langle x_\alpha(t) \mid \alpha \in \roots_{\Spin_n}, t\in \fieldsep \rangle.
\end{align*}
Here we are using the language of Chevalley generators as in \cite{Stein} to describe $\Spin_n$.  In this language, the $\fieldsep$ points of a group scheme $\G$ are generated by elements $x_\alpha(t)$ where $\alpha \in \roots_{\G}$ is a member of the root system of the group and $t \in \fieldsep$. Additional significant elements are named
\[
w_\alpha(t) := x_\alpha(t) x_{-\alpha}(-t^{-1}) x_\alpha(t) \text{ and } h_\alpha(t):= w_\alpha(t)w_\alpha(-1) \quad \text{for }  t \in \fieldsep^\times.
\]
The generators of $\Spin_n$ satisfy the relations of appendix \ref{Spin app}. To see the relationship between these groups in the language of Chevalley generators we use \cite[Lemma 28(d)]{Stein} to compute the center of $\Spin_n$. The center of $\Spin_{2n}$ when $n$ is even can also be found in \cite[Example 8.6]{GQ-M}. Noting that $\Spin_{2n}$ is type D$_n$ and $\Spin_{2n+1}$ is type B$_n$, in each case let $\roots$ be the root system and choose the standard simple system of roots $\Delta=\{\alpha_1,\ldots,\alpha_n\} \subset \Phi$ as in \cite{Bou4-6}. Then
\[
Z(\Spin_{2n}) \cong \begin{cases} \mu_2\times\mu_2 & n \text{ even}\\
 \mu_4 & n \text{ odd}
\end{cases} \text{ and }Z(\Spin_{2n+1}) \cong \mu_2.
\]
The $\fieldsep$-points are
\begin{align*}
&Z(\Spin_{2n}(\fieldsep)) = \begin{cases}\{ 1, \xi_1,  h_{\alpha_{n-1}}(-1)h_{\alpha_n}(-1), \xi_2 \} & n \text{ even}\\
\{ 1, \zeta, \zeta^2, \zeta^3 \} & n \text{ odd}\\
\end{cases}\\
&Z(\Spin_{2n+1}(\fieldsep)) = \{ 1, h_{\alpha_n}(-1)\}.
\end{align*}
where
\begin{align*}
\xi_1 &= \prod_{k=1 \atop k \text{ odd}}^{n-1} h_{\alpha_k}(-1), & \zeta &= \prod_{k=1 \atop k \text{ odd}}^{n-2} h_{\alpha_k}(-1) \cdot h_{\alpha_{n-1}}(i) h_{\alpha_n}(-i) \\
\xi_2&=h_{\alpha_{n-1}}(-1)h_{\alpha_n}(-1)\xi_1
\end{align*}
and $i\in \fieldsep$ is an element such that $i^2=-1$. Note that since the $h_\alpha(t)$ commute with one another and are multiplicative in their arguments, $\zeta^2 = h_{\alpha_{n-1}}(-1)h_{\alpha_n}(-1)$. Therefore we have that
\begin{align*}
\SO_{2n}(\fieldsep) &\cong \Spin_{2n}(\fieldsep)/\{1, h_{\alpha_{n-1}}(-1)h_{\alpha_n}(-1)\} \\
\PSO_{2n} &\cong \Spin_{2n}/Z(\Spin_{2n}) \\
\SO_{2n+1} &\cong \Spin_{2n+1}/Z(\Spin_{2n+1}) \cong \PSO_{2n+1}.
\end{align*}
Hence these groups can also be described by the Chevalley generators $x_\alpha(t)$ subject to the usual relations in $\Spin_n$ with the additional relations imposed by the quotients. The translation between the Chevalley language and the matrix language within $\SO_n$ is given in appendix \ref{Spin app}. When $n$ is even, $\Spin_{2n}$ has additional central subgroups. The \emph{half-spin group scheme} is defined over $\fieldsep$ by
\[
\HSpin_{2n}(\fieldsep) = \Spin_{2n}(\fieldsep)/\{1, \xi_1\} \cong \Spin_{2n}(\fieldsep)/\{1, \xi_2\}.
\]

We also consider a symplectic involution on $\Mat_{2n}(\field)$ given by
\[
\begin{tabular}{cc}
$\begin{aligned}
\sympinv_{2n} \colon \Mat_{2n}(\field) &\to \Mat_{2n}(\field) \\
A &\mapsto -\sympform_{2n} A^T \sympform_{2n}
\end{aligned}$ & \text{ where } $\sympform_{2n}=\begin{bmatrix}0 & \orthoform_n \\ -\orthoform_n & 0 \end{bmatrix}$
\end{tabular}
\]
and the associated group schemes $\PSP_{2n}(\sympinv_{2n})$ and $\SP_{2n}(\sympinv_{2n})$. Again we use the unadorned $\PSP_{2n}$ and $\SP_{2n}$ when the symplectic involution in question is the above $\sympinv_{2n}$. These have $\fieldsep$-points
\begin{align*}
\SP_{2n}(\fieldsep) &= \{ A \in \Mat_{2n}(\fieldsep) \mid A\cdot (\sympinv_{2n}\otimes 1)(A)=I \} \\
\PSP_{2n}(\fieldsep) &= \SP_{2n}(\fieldsep)/\{I,-I\} 
\end{align*}
and $Z(\SP_{2n}) \cong \mu_2$. We will also work with the symplectic groups via the language of Chevalley generators. The translation between the Chevalley language and the matrix language within $\SP_{2n}$ is given in appendix \ref{SP app}. We note here that in this language the center of the symplectic group is 
\[
Z(\SP_{2n}(\fieldsep)) = \{ I, \prod_{i=1 \atop i \text{ odd}}^n h_{\alpha_i}(-1)\}.
\]
Finally, we will obtain the $\field$-points of all these groups by considering the fixed points of the $\Galgroup = \Gal(\fieldsep/\field)$ action on the $\fieldsep$-points described above. The action of $\Galgroup$ on Chevalley generators is included in appendix \ref{SP app} and \ref{Spin app}.

\subsection{Kronecker Tensor Product}
The titular Kronecker tensor product map, on the level of matrices over a field extension $\Efield/\field$, is the universal map $\Mat_n(\Efield)\times \Mat_m(\Efield) \to \Mat_n(\Efield)\otimes \Mat_m(\Efield)\cong \Mat_{nm}(\Efield)$. These maps restrict to produce the following two group scheme homomorphisms.
\begin{align*}
\rho_{\SO} \colon &\SO_n \times \SO_m\to \SO_{nm} \\
\rho_0 \colon &\SP_{2n}\times \SP_{2m} \to \SO_{4nm}(\sympinv_{2n}\otimes \sympinv_{2m}).
\end{align*}
Rather than consider the second homomorphism, we will compose it with a conjugation to obtain
\begin{align*}
\rho_{\SP} \colon \SP_{2n}\times \SP_{2m} &\to \SO_{4nm} \\
(A,B) &\mapsto P^{-1}\rho_0(A,B)P
\end{align*}
where
\[\tiny
P= \begin{array}{c}
\left[\begin{array}{ccc;{2pt/2pt}ccc}
J & & & & & \\
 & \ddots & & & & \\
 & & J & & & \\ \hdashline[2pt/2pt]
  & & & K & & \\
  & & & & \ddots & \\
  & & & & & K \\
\end{array}\right] \\
\hexbrace{6em}{n}\quad\hexbrace{6em}{n}
\end{array},\scriptsize \text{ with }  J=\begin{bmatrix} I_m & 0 \\ 0 & -\orthoform_m \end{bmatrix}, K=\begin{bmatrix}\orthoform_m & 0 \\ 0 & I_m\end{bmatrix}.
\]

\subsection{Rost Multipliers}
Here we introduce the notion of the Rost multipliers of a homomorphism between linear algebraic groups following \cite{GMS}. Rost multipliers are integers that describe the homomorphism's induced map between quadratic invariants as follows. Let $\G$ be a split semisimple linear algebraic group and let $\T\subset \G$ be a split maximal torus.
\[
\T \cong \Gmult^n \;\text{ and }\; \T^* = \Hom_\field(\T,\Gmult) \cong \Z^n.
\]
Since $\T$ acts diagonalizably on the Lie algebra of $\G$ via the adjoint action, a copy of the root system of $\G$, denoted $\roots$, lies within $\T^*$. Let $W$ be the Weyl group of $\roots$. The action of $W$ on $\roots$ extends naturally to $\T^*$ and then also to the degree $2$ elements of the symmetric tensor product, $S^2(\T^*)$. The invariant elements, $\Sym^2(\T^*)^W$, are integral $W$-invariant quadratic forms on $\T^*\otimes \R$. These are the aforementioned quadratic invariants. When $\G$ is simple, $\Sym^2(\T^*)^W \cong \Z\langle q \rangle$ is an infinite cyclic group generated by an element $q$ called the \emph{normalized Killing form} of $\G$. When $\G$ is semisimple these invariants form a free group generated by the normalized Killing forms of the simple components.

This construction has functorial properties, \cite[pg.119]{GMS}. If $\varphi \colon \G \rightarrow \Hh$ is a homomorphism of split semisimple linear algebraic groups, then there is an induced map on characters, $\varphi^*\colon \T_{\Hh}^* \rightarrow \T_{\G}^*$, which in turn induces a map on quadratic invariants
\[
\varphi^\dagger \colon \Sym(\T_{\Hh}^*)^{W_{\Hh}} \rightarrow \Sym(T_{\G}^*)^{W_{\G}}
\]
by extending algebraically. Because $\varphi^\dagger$ is a homomorphism of free groups, the image of each generator is an integral combination of generators in the codomain, and these integers are called the \emph{Rost multipliers} of $\varphi$. 

\section{Garibaldi's Example, $\PSPp_2\times\PSPp_8\hookrightarrow \HSpinp_{16}$}\label{Garibaldi_Example}
In \cite{Gar}, Garibaldi considered the map
\[
\begin{array}{lc}
\begin{aligned}
\varphi \colon \SP_2(\field) \times \SP_8(\field) &\to \SO_{16}(\field) \\
(A,B) &\mapsto P^{-1}\rho(A,B)P
\end{aligned} &  \text{ where } 
P = \begin{bmatrix} I_4 & 0 & 0 & 0 \\
0 & 0 & \orthoform_4 & 0 \\
0 & -\orthoform_4  & 0 & 0\\
0 & 0 & 0 & I_4 
\end{bmatrix}.
\end{array}
\]
This $P$ differs slightly from the one printed in \cite{Gar} but later calculations there agree with this $P$. We use the notational shorthand $\PSPp_2 = \PSP_2(\field)$ and likewise for other groups where appropriate. Garibaldi described the restriction of $\varphi$ to maximal tori using the language of Chevalley groups and then noted that there is a lifting $\phi \colon \SPp_2 \times \SPp_8 \rightarrow \Spinp_{16}$ which acts analogously on Chevalley generators. The lifting's restriction to maximal tori, and in particular to the center, is described in the following tables.
\begin{align*}
\begin{array}{c|c}
h\in\SPp_2\times \SPp_8 & \phi(h)\in \Spinp_{16} \\ \hline
(h_1(t),I) & h_{1}(t)h_{2}(t^2)h_{3}(t^3)h_{4}(t^4)h_{5}(t^3)h_{6}(t^2)h_{7}(t) \\
(I,h_1(u)) & h_{1}(u) h_{7}(u^{-1}) \\
(I,h_2(u)) & h_{2}(u) h_{6}(u^{-1}) \\
(I,h_3(u)) & h_{3}(u) h_{5}(u^{-1}) \\
(I,h_4(u)) & h_{4}(u) h_{5}(u^2) h_{6}(u^2) h_{7}(u) h_{8}(u)
\end{array}\\[1em] 
\begin{array}{c|c}
h \in Z(\SPp_2 \times \SPp_8) & \phi(h) \in \Spinp_{16} \\ \hline
(I,I) & 1 \\
(h_1(-1),I) & h_{1}(-1) h_{3}(-1) h_{5}(-1) h_{7}(-1) = \xi_1 \\
(I,h_{1}(-1)h_{3}(-1)) & h_{1}(-1) h_{3}(-1) h_{5}(-1) h_{7}(-1) = \xi_1 \\
(h_1(-1), h_{1}(-1) h_{3}(-1) & 1 \\
\end{array}
\end{align*}
Therefore $\phi$ induces an  injection $(\SPp_2\times \SPp_8 )/ Z(\SPp_2\times \SPp_8)\hookrightarrow \Spinp_{16}/\{1, \xi_1 \}$ which is the desired map $\PSPp_2 \times \PSPp_8 \hookrightarrow \HSpinp_{16}$.

\section{Chevalley Generator Descriptions of Kronecker Tensor Products}\label{sec:liftings}
Following the method of the previous example, we aim to explicitly describe Kronecker tensor product maps between simply connected groups using the language of Chevalley generators. To begin, we apply \cite[Proposition 2.24(i)]{BT} to the maps, 
\[
\rho_{\SP} \colon \SP_{2n}\times \SP_{2m} \to \SO_{4nm}
\]
\[
\Spin_n\times\Spin_m \twoheadrightarrow \SO_n\times \SO_m \overset{\rho_{\SO}}{\to} \SO_{nm}
\]
By doing so we obtain unique maps $\phi_{\SP}$ and $\phi_{\Spin}$ making the following diagrams commute.
\[
\begin{tikzcd}
 & \Spin_{4nm} \arrow[two heads]{d} \\
 \SP_{2n}\times \SP_{2m} \arrow{ur}{\phi_{\SP}} \arrow{r}{\rho_{\SP}} & \SO_{4nm}
\end{tikzcd}
\begin{tikzcd}
\Spin_n\times \Spin_m \arrow[two heads]{d} \arrow{r}{\phi_{\Spin}} & \Spin_{nm} \arrow[two heads]{d} \\
\SO_n\times \SO_m \arrow{r}{\rho_{\SO}} & \SO_{nm}
\end{tikzcd}
\]
Since various properties of $\phi_{\Spin}$ depend on the parities of $n$ and $m$, we also denote it by $\phi_{n,m}$. On the level of $\fieldsep$-points these groups are described by the Chevalley generators given in appendix \ref{SP app} and \ref{Spin app}, so we may describe the maps in terms of Chevalley generators also. The images for the maps $\rho_{\SP}(\fieldsep)$ and $\rho_{\SO}(\fieldsep)$ can be computed explicitly using the matrix representations of $\SP_{2n}(\fieldsep)$ and $\SO_n(\fieldsep)$. These images are recorded in appendix \ref{Images}. It turns out that the images of the lifted maps $\phi_{\SP}$ and $\phi_{\Spin}$ are described with analogous products of Chevalley generators. To verify this we use the following collection of identities which we have included since it is difficult to find such descriptions of $\Spin$ in the literature.

\begin{lem}\label{Spin mult}
Let $n\geq 2$ and let $\Phi$ be the root system of $\Spin_n$. The following relations hold in $\Spin_n(\fieldsep)$ for all $\alpha,\beta \in \Phi$.
\begin{itemize}
\item[(1)] $w_\alpha(t)x_{\beta}(u) w_\alpha(-t) = x_{r_\alpha(\beta)}(ct^{-\langle \beta,\alpha \rangle}u)$
\item[(2)] $w_\alpha(t)w_\beta(u)w_\alpha(-t)= w_{r_\alpha(\beta)}(ct^{-\langle \beta,\alpha \rangle}u)$
\item[(3)] $w_\alpha(t) = w_{-\alpha}(-t^{-1})$
\item[(4)] $w_\alpha(tu) = w_\alpha(t)w_\alpha(-1)w_\alpha(u)$
\item[(5)] $w_\alpha(t) h_\beta(u) w_\alpha(-t) = h_{r_\alpha(\beta)}(u)$
\item[(6)] $h_\alpha(t)^{-1}  = h_{-\alpha}(t)$
\item[(7)] $w_\alpha(t)^2 = h_\alpha(-1)$
\end{itemize}
where $r_\alpha(\beta)=\beta - \tfrac{2(\alpha,\beta)}{(\alpha,\alpha)}\alpha$ is the reflection of the root $\beta$ across the hyperplane perpendicular to $\alpha$, and $c=\pm 1$ depends in each case on $\alpha$ and $\beta$.

Furthermore, if $\alpha,\beta \in \Phi$ are any two roots in type $D$ or two long roots in type $B$ such that $\alpha+\beta \in \Phi$ then
\begin{itemize}
\item[(8)] $h_\alpha(t)h_\beta(t)=h_{\alpha+\beta}(t)$.
\end{itemize}
In type $B$, if $\alpha,\beta \in \Phi$ are a long and short root respectively such that $\alpha+\beta \in \Phi$ then
\begin{itemize}
\item[(9)] $h_\alpha(t)h_\beta(t)=h_{\alpha+2\beta}(t)$.
\end{itemize}
Finally, in type $B$ if $\alpha,\beta \in \Phi$ are two short roots such that $\alpha+\beta \in \Phi$ then
\begin{itemize}
\item[(10)] $h_\alpha(t)h_\beta(t)=h_{\alpha+\beta}(t^2)$.
\end{itemize}
\end{lem}
\begin{proof}
We reference the identities (R-) given \cite[pg. 16]{Stein}. We first note that by (R1): $x_\alpha(t)x_\alpha(u)=x_\alpha(t+u)$, we have $w_\alpha(t)^{-1}=w_\alpha(-t)$. Then points 1),2), and 3) are justified by using the identities
\begin{align*} 
\text{(R7): }& w_\alpha(1)x_\beta(u)w_\alpha(1)^{-1} = x_{r_\alpha(\beta)}(cu) \text{ for some constant } c=\pm 1, \\
\text{(R8): }& h_\alpha(t)x_\beta(u)h_\alpha(t)^{-1} = x_\beta(t^{\langle \beta,\alpha \rangle}u)
\end{align*}
where $\langle \beta,\alpha \rangle = \tfrac{2(\beta,\alpha)}{(\alpha,\alpha)}$. Expanding the $h_\alpha$'s in (R8) and applying (R7) gives
\[
w_\alpha(t) x_{r_\alpha(\beta)}(cu) w_\alpha(-t) = x_\beta(t^{\langle \beta,\alpha \rangle}u).
\]
Inputting $r_\alpha(\beta)$ in the place of $\beta$, and $cu$ in the place of $u$ gives
\[
(1)\; w_\alpha(t)x_{\beta}(u) w_\alpha(-t) = x_{r_\alpha(\beta)}(ct^{-\langle \beta,\alpha \rangle}u).
\]
Applying such conjugations to the $x_\beta(u)$ and $x_{-\beta}(-u^{-1})$ in the expansion of $w_\beta(u)$ yields (2). Then (3) comes from (2) by conjugating $w_\alpha(t)$ with itself,
\[
w_\alpha(t)=w_\alpha(t)w_\alpha(t)w_\alpha(-t) = w_{-\alpha}(ct^{-2}t) = w_{-\alpha}(ct^{-1}).
\]
Then since the constants $c$ are the same in $\Spin$ as in $\SO$, calculations in $\SO$ show that for the above case $c=-1$, and so $w_\alpha(t)=w_{-\alpha}(-t^{-1})$.

Point (4) follows from \cite[Lemma 28(a)]{Stein}: $h_\alpha(tu)=h_\alpha(t)h_\alpha(u)$, by expanding both sides into $w$'s and cancelling the rightmost $w_\alpha(-1)$'s. 

This leads to point (5),
\begin{align*}
w_\alpha(t) h_\beta(u) w_\alpha(-t) &= w_\alpha(t)w_\beta(u)w_\alpha(-t)w_\alpha(t)w_\beta(-1)w_\alpha(-t) \\
&= w_{r_\alpha(\beta)}(ct^{-\langle \beta,\alpha \rangle}u)w_{r_\alpha(\beta)}(-ct^{-\langle \beta,\alpha \rangle}) \\
&= w_{r_\alpha(\beta)}(u)w_{r_\alpha(\beta)}(-1)w_{r_\alpha(\beta)}(ct^{-\langle \beta,\alpha \rangle})w_{r_\alpha(\beta)}(-ct^{-\langle \beta,\alpha \rangle})\\
&=h_{r_\alpha(\beta)}(u).
\end{align*}
Justifying (6) is an application of (2) and (3).
\begin{align*}
h_\alpha(t)^{-1} &= w_\alpha(-1)^{-1}w_\alpha(t)^{-1} = w_\alpha(1)w_\alpha(-t) = w_{-\alpha}(-1)w_{-\alpha}(t^{-1}) \\
&= w_\alpha(-(-1)^2 t^{-1})w_{-\alpha}(-1) = w_{-\alpha}(t)w_{-\alpha}(-1) = h_{-\alpha}(t). 
\end{align*}
Point (7) follows from (3),(4) and (6).
\begin{align*}
w_\alpha(t)^2 &= w_\alpha(t)w_\alpha(t) = w_\alpha(t)w_\alpha(-1)w_\alpha(-1)w_\alpha(-t) = h_\alpha(t)\big(w_\alpha(t)w_\alpha(1)\big)^{-1} \\
&= h_\alpha(t)\big( w_{-\alpha}(-t^{-1})w_{-\alpha}(-1)\big)^{-1} = h_\alpha(t)h_{-\alpha}(-t^{-1})^{-1} = h_\alpha(t)h_\alpha(-t^{-1}) \\
&= h_\alpha(-1).
\end{align*}
To see (8), consider $\alpha,\beta \in \Phi$, which are any roots in type $D$ or long roots in type $B$ such that $\alpha+\beta \in \Phi$. Without loss of generality they must be of the form $\alpha = \pm e_i + (-1)^m e_j, \beta = (-1)^{m+1}e_j \pm e_k$ with $i\neq k$. Therefore $-\langle \beta,\alpha \rangle = -\langle \alpha,\beta \rangle = 1$ and $r_\alpha(\beta)=r_\beta(\alpha)=\alpha+\beta$. Then using (2) and (5)
\begin{align*}
h_\alpha(t)h_\beta(t) &= w_\alpha(t)w_\alpha(-1)h_\beta(t) = w_\alpha(t)h_{\alpha+\beta}(t)w_\alpha(-1) \\
&= \big(w_\alpha(t)w_{\alpha+\beta}(t)w_\alpha(-t)\big)w_\alpha(t)w_\alpha(-1)\big(w_\alpha(1)w_{\alpha+\beta}(-1)w_\alpha(-1)\big)\\
&=w_\beta(ct^{-\langle \alpha,\alpha+\beta \rangle}t)h_\alpha(t) w_\beta(-c) = w_\beta(c)h_\alpha(t)w_\beta(-c) \\
&= h_{\alpha+\beta}(t)w_\beta(c)w_\beta(-c) = h_{\alpha+\beta}(t).
\end{align*}
For (9), when in type $B$ and $\alpha,\beta \in \Phi$ are a long and short root respectively such that $\alpha+\beta \in \Phi$ then without loss of generality they are of the form $\alpha=\pm e_i +(-1)^m e_j, \beta=(-1)^{m+1}e_j$. Therefore $-\langle \beta,\alpha\rangle = 1, -\langle \alpha,\beta \rangle = 2$ and so $r_\alpha(\beta)=\alpha+\beta, r_\beta(\alpha)=\alpha+2\beta$. The argument is then the same as the one above except at the second last step.
\[
h_\alpha(t)h_\beta(t)=\ldots=w_\beta(c)h_\alpha(t)w_\beta(-c) = h_{\alpha+2\beta}(t)w_\beta(c)w_\beta(-c)=h_{\alpha+2\beta}(t).
\]
Finally (10) follows from (6) and (9). When in type $B$ if $\alpha,\beta \in \Phi$ are two short roots such that $\alpha+\beta \in \Phi$ then $\alpha=\pm e_i, \beta= \pm e_j$ with $i\neq j$. Therefore $-\alpha+\beta$ is a long root and so
\[
h_\alpha(t)h_\beta(t) = h_\alpha(t)h_{-\alpha+\beta}(t)h_{\alpha-\beta}(t)h_\beta(t) = h_{\alpha+\beta}(t)h_{\alpha+\beta}(t) = h_{\alpha+\beta}(t^2).\qedhere
\]
\end{proof}

\begin{prop}\label{main theorem}
The group scheme homomorphisms $\phi_{\SP}$ and $\phi_{\Spin}$ act analogously on Chevalley generators as the homomorphisms $\rho_{\SP}$ and $\rho_{\SO}$ respectively. For example
\[
\rho_{2n,2m}((x_{e_i - e_j}(t),1) :=  \prod_{k=1}^{2m} x_{e_{(i-1)2m+k}-e_{(j-1)2m+k}}(t)
\]
where the $x_\alpha(t)$ appearing are generators of $\SO$, and therefore
\[
\phi_{2n,2m}((x_{e_i - e_j}(t),1) :=  \prod_{k=1}^{2m} x_{e_{(i-1)2m+k}-e_{(j-1)2m+k}}(t)
\]
where now the $x_\alpha(t)$ appearing represent generators of $\Spin$. All images of $\phi_{\SP}$ and $\phi_{\Spin}$ are obtained similarly.
\end{prop}
\begin{proof}
First, define homomorphisms 
\[
f\colon \SP_{2n}\times\SP_{2m}\to \Spin_{4nm} \text{ and } g\colon \Spin_n\times \Spin_m \to \Spin_{nm}
\]
to act analogously to $\rho_{\SP}$ and $\rho_{\SO}$ respectively on $\fieldsep$-points as recorded in appendix \ref{Images}. Since these maps are between simply connected groups, to verify they are well-defined it is sufficient via \cite[Theorem 8]{Stein} to check that the image of any $x_\alpha(t)$ is again additive in $t$, the image of any $h_\alpha(t)$ remains multiplicative in $t$, and that the commutator relations given in the appendices are preserved. This can be verified using the relations in $\Spin$ given in appendix \ref{Spin app} and the relations of lemma \ref{Spin mult}. We note that $f$ and $g$ both preserve the action of the Galois group $\Gal(\fieldsep/\field)$ and therefore descend to full group scheme homomorphisms. Since the natural projection $\Spin_n \twoheadrightarrow \SO_n$ sends a Chevalley generator $x_\alpha(t)\in \Spin_n(\fieldsep)$ to the analogous generator $x_\alpha(t)\in\SO_n(\fieldsep)$, $f$ and $g$ clearly make the following diagrams commute.
\[
\begin{tikzcd}
 & \Spin_{4nm} \arrow[two heads]{d} \\
 \SP_{2n}\times \SP_{2m} \arrow{ur}{f} \arrow{r}{\rho_{\SP}} & \SO_{4nm}
\end{tikzcd} \quad
\begin{tikzcd}
\Spin_{n}\times \Spin_{m} \arrow[two heads]{d} \arrow{r}{g} & \Spin_{nm} \arrow[two heads]{d} \\
\SO_{n}\times \SO_{m} \arrow{r}{\rho_{\SO}} & \SO_{nm}
\end{tikzcd}
\]
Finally, since \cite[Proposition 2.24(i)]{BT} produces unique maps, we must have that $f=\phi_{\SP}$ and $g=\phi_{\Spin}$.
\end{proof}
\begin{rem}
We note that the maps $\phi_{\SP}$ and $\phi_{\Spin}$ also exist when $\field$ is characteristic 2, however it is unclear to us whether the techniques used in proposition \ref{main theorem} apply in that case.
\end{rem}

\begin{cor}\label{cor:max tori}
The maps $\phi_{\SP},\phi_{\Spin}$ restrict to maps between maximal tori. For long roots ($\pm e_i \pm e_j$ or $\pm2e_i$) the images of $h_\alpha(t)$ are products of $h's$ analogous to the product of $x's$ in the image of $x_\alpha(t)$. For short roots in type B, the images are 
\begin{align*}
\phi_{2n,2m+1}(1,h_{e_i}(u)) &=\prod_{k=0}^{n-1} h_{e_{[k]+i}-e_{[k]+\overline{i}}}(u^2)\\
\phi_{2n+1,2m+1}(h_{e_i}(t),1) &= \prod_{k=1}^{2m+1} h_{e_{[i-1]+k}}(t) \\
\phi_{2n+1,2m+1}(1,h_{e_i}(u)) &= h_{e_{[n]+i}}(u)\prod_{k=0}^{n-1} h_{e_{[k]+i}}(u) h_{e_{[k]+\overline{i}}}(u^{-1})
\end{align*}
where $[y]=(2m+1)y$ and $\overline{y}=2m+2-y$.
\end{cor}
\begin{proof}
These images can be computed using the explicit descriptions of $\phi_{\SP}$ and $\phi_{\Spin}$ given in proposition \ref{main theorem}. For long roots, the factors in the image of $x_\alpha(t)$ pairwise commute. Therefore $w_\alpha(t)$ maps to an analogous product of $w$'s, and in turn $h_\alpha(t)$ maps to an analogous product of $h$'s. For short roots in type B the image of $x_\alpha(t)$ contains factors which do not commute. In these cases the image of $h_\alpha(t)$ can be computed using additional identities such as those in \cite{Stein} and lemma \ref{Spin mult}.
\end{proof}

\section{Induced Maps Via Central Quotients}\label{sec:induced}
Now that we have our desired liftings we wish to track their behaviour on central elements in order to identify induced maps between central quotients. In particular, we are interested in maps involving $\HSpin$. 

\begin{prop}\label{prop: PSp induced}
Let at least one of $n$ and $m$ be even. Then there exists an injection of group schemes $\phi_{\SP}'$ making the following diagram commute.
\[
\begin{tikzcd}
\SP_{2n}\times \SP_{2m} \arrow[twoheadrightarrow]{d} \arrow{r}{\phi_{\SP}} & \Spin_{4nm} \arrow[twoheadrightarrow]{d}\\
\PSP_{2n}\times\PSP_{2m} \arrow[hookrightarrow]{r}{\phi_{\SP}'} & \HSpin_{4nm}
\end{tikzcd}
\]
\end{prop}
\begin{proof}
Since the kernel of the map $\rho_{\SP}$ is a central subgroup of $\SP_{2n}\times\SP_{2m}$, and $\rho_{\SP}$ factors through $\Spin_{4nm}$ via $\phi_{\SP}$, we must have that $\ker(\phi_{\SP})$ is contained within the center as well. By direct computation using the results of corollary \ref{cor:max tori} and the description of the centers of $\SP$ and $\Spin$ in section \ref{sec:split groups}, we see that when at least one of $n,m$ is even $\phi_{\SP}$ behaves on elements of the center of $\SP_{2n}\times \SP_{2m}$ as
\[
(I,I)\mapsto 1,\quad (-I,I)\mapsto \xi_1,\quad (I,-I)\mapsto \xi_1,\quad (-I,-I)\mapsto 1.
\]
Thus the kernel of the composition $\SP_{2n}\times \SP_{2m}\rightarrow \Spin_{4nm}\rightarrow \HSpin_{4nm}$ is the center of $\SP_{2n}\times\SP_{2m}$. The first isomorphism theorem then yields the desired injection.
\end{proof}

\begin{prop}\label{prop: PSO induced}
Let at least one of $n$ and $m$ be even. Then there exists an injection of group schemes $\phi_{2n,2m}'$ making the following diagram commute.
\[
\begin{tikzcd}
\Spin_{2n}\times\Spin_{2m} \arrow[twoheadrightarrow]{d} \arrow{r}{\phi_{2n,2m}} & \Spin_{4nm} \arrow[twoheadrightarrow]{d} \\
\PSO_{2n}\times \PSO_{2m} \arrow[hookrightarrow]{r}{\phi_{2n,2m}'} & \HSpin_{4nm}
\end{tikzcd}
\]
\end{prop}
\begin{proof}
The kernel of $\phi_{2n,2m}$ is a subset of the kernel of the composition $\Spin_{2n}\times\Spin_{2m}\to \Spin_{4nm}\to \SO_{4nm}$ which is the same map as the composition $\Spin_{2n}\times\Spin_{2m}\to \SO_{2n}\times\SO_{2m} \to \SO_{4nm}$ by construction. The kernel of $\rho_{2n,2m}$ is a central subgroup of $\SO_{2n}\times\SO_{2m}$ and the inverse of $Z(\SO_{2n}\times\SO_{2m})$ under the natural projection is $Z(\Spin_{2n}\times\Spin_{2m})$. Thus the kernel of $\phi_{2n,2m}$ is a central subgroup of $\Spin_{2n}\times\Spin_{2m}$. Using the description of the center of $\Spin$ in section \ref{sec:split groups} and corollary \ref{cor:max tori} we compute that $\phi_{2n,2m}$ behaves on generators of the center as follows, depending on the parities of $n$ and $m$.
\[
\begin{tabular}{r@{}l|r@{}l|r@{}l}
\multicolumn{2}{c}{$n \text{ even } \atop m \text{ even}$}&\multicolumn{2}{c}{$n \text{ odd } \atop m \text{ even}$}&\multicolumn{2}{c}{$n \text{ even } \atop m \text{ odd}$}\\
$(\xi_1,1)$ &$\mapsto \xi_1$ & $(\zeta,1)$ & $\mapsto \xi_1$ & $(\xi_1,1)$ & $\mapsto\xi_1$  \\
$(\xi_2,1)$ &$\mapsto \xi_1$ & $(1,\xi_1)$ & $\mapsto \xi_1$ & $(\xi_2,1)$ & $\mapsto \xi_1$ \\
$(1,\xi_1)$ &$\mapsto \xi_1$ & $(1,\xi_2)$ & $\mapsto \xi_1$ & $(1,\zeta)$ & $\mapsto \xi_1$ \\ 
$(1,\xi_2)$ &$\mapsto \xi_1$ & & & \\
\end{tabular}
\]
In all cases we see that the kernel of the composition 
\[
\Spin_{2n}\times\Spin_{2m} \to \Spin_{4nm}\to\HSpin_{4nm}
\]
is the center of $\Spin_{2n}\times\Spin_{2m}$. The first isomorphism theorem then yields the desired injection.
\end{proof}

\begin{prop}\label{prop:HSpin induced}
Let $n$ be even. Then there exists an injection of group schemes $\phi_{2n,2m+1}'$ making the following diagram commute.
\[
\begin{tikzcd}
\Spin_{2n}\times\Spin_{2m+1} \arrow[twoheadrightarrow]{d} \arrow{r}{\phi_{2n,2m+1}} & \Spin_{2n(2m+1)} \arrow[twoheadrightarrow]{d} \\
\HSpin_{2n}\times \SO_{2m+1} \arrow[hookrightarrow]{r}{\phi_{2n,2m+1}'} & \HSpin_{2n(2m+1)}
\end{tikzcd}
\]
\end{prop}
\begin{proof}
By similar arguments as in the proof of proposition \ref{prop: PSO induced} above, the kernel of $\phi_{2n,2m+1}$ is a central subgroup of $\Spin_{2n}\times\Spin_{2m+1}$. Using the description of the center of $\Spin$ in section \ref{sec:split groups} and corollary \ref{cor:max tori} we compute that $\phi_{2n,2m+1}$ behaves on generators of the center of $\Spin_{2n}\times\Spin_{2m+1}$ as
\[
(\xi_1,1) \mapsto \xi_1, \quad (\xi_2,1)\mapsto \xi_2, \quad (1,h_{\alpha_n}(-1)) \mapsto 1.
\]
Thus the kernel of the composition 
\[
\Spin_{2n}\times\Spin_{2m+1}\to\Spin_{2n(2m+1)}\to\HSpin_{2n(2m+1)}
\]
is $\{1,\xi_1\}\times Z(\Spin_{2m+1})$. The first isomorphism theorem then yields the desired injection.
\end{proof}

\section{Computing Rost Multipliers}\label{sec:Rost}
In what follows we compute the Rost multipliers of the maps constructed in section \ref{sec:induced}. For each of our groups, we choose the standard simple system of roots $\{\alpha_1,\ldots,\alpha_n\}$ as in \cite{Bou4-6} where $n$ is the rank of the group. and then consider the maximal torus given by
\[
\T(\field)=\langle \prod_{i=1}^{n} h_{\alpha_i}(t_i) \mid t_i \in \field^\times \rangle.
\]
In all types, for $\sigma\in \Galgroup=\Gal(\fieldsep/\field)$ the action on $h$'s is by $\sigma(h_\alpha(t))=h_\alpha(\sigma(t))$. Therefore the group $\T(\field)$ consists of products of $h$'s with arguments from $\field$. The groups $\T^*$ are then
\[
\begin{array}{l|l}
\begin{array}{l}\T^*_{\Spin_{2n}} \\ \T^*_{\Spin_{2n+1}} \end{array} & \Z\langle e_1,\ldots,e_{n-1}, \frac{1}{2}(e_1+\ldots+e_n) \rangle \\ \hline
\begin{array}{l}\T^*_{\SO_{2n}} \\ \T^*_{\SO_{2n+1}} \\ \T^*_{\SP_{2n}} \end{array} & \Z\langle e_1,\ldots,e_n \rangle \\ \hline
\T^*_{\HSpin_{2n}} & \{ \sum_{i=1}^{n-1}c_i e_i + c_n \frac{1}{2}(e_1+\ldots e_n) \mid c_i\in\Z,\sum_{i=1}^{n-1} c_i \text{ even}\} \\ \hline
\begin{array}{l} \T^*_{\PSO_{2n}} \\ \T^*_{\PSP_{2n}} \end{array} & \{ \sum_{i=1}^n c_i e_i \mid c_i\in\Z, \sum_{i=1}^n c_i \text{ even}\}.
\end{array}
\]

where the characters act as follows by type.
\[
\begin{array}{r|ll}
D_n & e_1(\prod h_{\alpha_i}(t_i))=t_1\\
& e_j(\prod h_{\alpha_i}(t_i))=t_{j-1}^{-1}t_j &\text{ for } 2\leq j\leq n-2 \text{ or } j=n\\
& e_{n-1}(\prod h_{\alpha_i}(t_i))=t_{n-2}^{-2}t_{n-1}t_n \\
& \frac{1}{2}(e_1+\ldots+e_n)(\prod h_{\alpha_i}(t_i))=t_n \vspace{3pt}\\ \hline 
B_n & e_1(\prod h_{\alpha_i}(t_i))=t_1 \\
 & e_j(\prod h_{\alpha_i}(t_i))=t_{j-1}^{-1}t_j &\text{ for } 2\leq j\leq n-1\\
 & e_n(\prod h_{\alpha_i}(t_i))=t_{n-1}^{-1}t_n^2 \\
 &\frac{1}{2}(e_1+\ldots+e_n)(\prod h_{\alpha_i}(t_i))=t_n \vspace{3pt}\\ \hline
C_n & e_1(\prod h_{\alpha_i}(t_i))=t_1 \\
 & e_j(\prod h_{\alpha_i}(t_i))=t_{j-1}^{-1}t_j &\text{ for } 2\leq j\leq n\\
\end{array}
\]
Each $\T^*$ contains the root system of its group in the usual way and therefore we can compute the normalized Killing forms by finding the smallest multiple of $\sum_{\alpha\in \Phi}\alpha^2$ present in $\Sym^2(\T^*)$. This yields the following.
\[
\begin{array}{c|c}
q_{\SO_{2n}}=q_{\SO_{2n+1}}=q_{\SP_{2n}} & q_{\PSP_{2n}} = q_{\PSO_{2n}}\\ \hline
\sum_{i=1}^n e_i^2 & \begin{array}{rl} \sum_{i=1}^n e_i^2 & n\cong 0 \pmod{4} \\ 2\sum_{i=1}^n e_i^2 & n \cong 2 \pmod{4} \\ 4\sum_{i=1}^n e_i^2 & n \cong 1,3 \pmod{4} \end{array}
\end{array}
\]
\[
\begin{array}{c|c}
q_{\Spin_{2n}}=q_{\Spin_{2n+1}} & q_{\HSpin_{4n}} \\ \hline
\frac{1}{2}\sum_{i=1}^n e_i^2 & \begin{array}{rl} \frac{1}{2}\sum_{i=1}^{2n} e_i^2 & n\cong 0 \pmod{4} \\ \sum_{i=1}^{2n} e_i^2 & n \cong 2 \pmod{4} \\ 2\sum_{i=1}^{2n} e_i^2 & n \cong 1,3 \pmod{4} \end{array}
\end{array}
\]
Since these are all rational multiples of $q_{\SO}=q_{\SP}$ it will be sufficient to describe the images $\rho_{\SP}^*(q_{\SP})$ and $\rho_{\SO}^*(q_{\SO})$ where $\rho_{\SP}\colon \SP_{2n}\times\SP_{2m}\to\SO_{4nm}$ and $\rho_{\SO}\colon \SO_{n}\times\SO_{m}\to\SO_{nm}$ are the tensor product maps from section \ref{sec:liftings}. From these, all other images can be extrapolated.

\begin{lem}\label{lem:SP restriction}
Let $\rho_{\SP}$ and $\phi_{\SP}$ be the split version of the maps from section \ref{sec:liftings}. Then the map $\rho_{\SP}^*\colon \T^*_{\SO_{4nm}} \rightarrow \T^*_{\SP_{2n}}\oplus T^*_{\SP_{2m}}$ is the restriction of the map $\phi_{\SP}^*\colon \T^*_{\Spin_{4nm}} \rightarrow \T^*_{\SP_{2n}}\oplus T^*_{\SP_{2m}}$ to the subgroup $\T^*_{\SO_{4nm}}\leq  \T^*_{\Spin_{4nm}}$.
\end{lem}
\begin{proof}
Restricting the commutative diagram on the left to maximal tori and then dualizing yields the adjacent diagram on the right,
\[
\begin{tikzcd}
 & \Spin_{4nm} \arrow[twoheadrightarrow]{d} \\
\SP_{2n} \times \SP_{2m} \arrow{ur}{\phi_{\SP}} \arrow{r}{\rho_{\SP}} & \SO_{4nm}
\end{tikzcd}
\begin{tikzcd}
 & \T^*_{\Spin_{4nm}} \arrow{dl}[swap]{\phi_{\SP}^*} \\
\T^*_{\SP_{2n}} \oplus \T^*_{\SP_{2m}} & \T^*_{\SO_{4nm}} \arrow[hookrightarrow]{u} \arrow{l}[swap]{\rho_{\SP}^*}
\end{tikzcd}
\]
from which the claim becomes clear. $\phi_{\SP}^*$ agrees with $\rho_{\SP}^*$ on the subgroup $\T^*_{\SO_{4nm}}$.
\end{proof}

\begin{lem}\label{lem:Spin restriction}
Let $\rho_{\SO}$ and $\phi_{\Spin}$ be the split version of the maps from section \ref{sec:liftings}. Then the map $\rho_{\SO}^*\colon \T^*_{\SO_{nm}} \rightarrow \T^*_{\SO_n}\oplus T^*_{\SO_m}$ can be identified with the restriction of the map $\phi_{\Spin}^*\colon \T^*_{\Spin_{nm}} \rightarrow \T^*_{\Spin_n}\oplus T^*_{\Spin_m}$ to the subgroup $\T^*_{\SO_{nm}}\leq  \T^*_{\Spin_{nm}}$.
\end{lem}
\begin{proof}
Restricting the commutative diagram on the left to maximal tori and then dualizing yields the adjacent diagram on the right.
\[
\begin{tikzcd}
\Spin_n \times \Spin_m \arrow{r}{\phi_{\Spin}} \arrow[twoheadrightarrow]{d} & \Spin_{nm} \arrow[twoheadrightarrow]{d} \\
\SO_n\times \SO_m \arrow{r}{\rho_{\SO}} & \SO_{nm}
\end{tikzcd}
\begin{tikzcd}
\T^*_{\Spin_n}\oplus \T^*_{\Spin_m} & \T^*_{\Spin_{nm}} \arrow{l}[swap]{\phi_{\Spin}^*} \\
\T^*_{\SO_n}\oplus \T^*_{\SO_m} \arrow[hookrightarrow]{u} & \T^*_{\SO_{nm}} \arrow[hookrightarrow]{u} \arrow{l}[swap]{\rho_{\SO}^*}
\end{tikzcd}
\]
Since the diagram is commutative we see that $\phi_{\Spin}^*$ maps elements of the subgroup $\T^*_{\SO_{nm}}\leq \T^*_{\Spin_{nm}}$ into the subgroup $\T^*_{\SO_{n}}\oplus\T^*_{\SO_{m}} \leq \T^*_{\Spin_n}\oplus \T^*_{\Spin_m}$. Therefore we may identify $\rho_{\SO}^*$ with this restriction of $\phi_{\Spin}^*$ to $\T^*_{\SO_{nm}}$.
\end{proof}

\begin{prop}\label{prop:SP multipliers}
The Rost multipliers of $\phi_{\SP}\colon \SP_{2n}\times\SP_{2m}\to \Spin_{4nm}$ are $(m,n)$. That is
\[
\phi_{\SP}^\dagger(q_{\Spin_{4nm}})=(mq_{\SP_{2n}}, nq_{\SP_{2m}}).
\]
\end{prop}
\begin{proof}
As a result of \cite[Proposition 7.9(5)(b)]{GMS} the Rost multipliers of $\rho_{\SP}$ are $(2m,2n)$. That is
\[
\rho_{\SP}^\dagger(q_{\SO_{4nm}})=(2mq_{\SP_{2n}},2nq_{\SP_{2m}}).
\]
Therefore using lemma \ref{lem:SP restriction} and the description of the normalized Killing forms above, we can compute
\begin{align*}
\phi_{\SP}^\dagger(q_{\Spin_{4nm}})&=\phi_{\SP}^\dagger(\frac{1}{2}q_{\SO_{4nm}}) = \frac{1}{2}\rho_{\SP}^\dagger(q_{\SO_{4nm}}) \\
&= \frac{1}{2}(2mq_{\SP_{2n}},2nq_{\SP_{2m}}) = (mq_{\SP_{2n}},nq_{\SP_{2m}}).
\end{align*}
Hence the Rost multipliers of $\phi_{\SP}$ are $(m,n)$.
\end{proof}

\begin{prop}\label{prop:Spin multipliers}
The Rost multipliers of $\phi_{n,m}\colon \Spin_{n}\times\Spin_{m} \to \Spin_{nm}$ are $(m,n)$. That is
\[
\phi_{n,m}^\dagger(q_{\Spin_{nm}})=(mq_{\Spin_{n}}, nq_{\Spin_{m}}).
\]
\end{prop}
\begin{proof}
As a result of \cite[Proposition 7.9(5)(b)]{GMS} the Rost multipliers of $\rho_{n,m}$ are $(m,n)$. That is
\[
\rho_{n,m}^\dagger(q_{\SO_{nm}})=(mq_{\SO_{n}},nq_{\SO_{m}}) \\
\]
Therefore by lemma \ref{lem:Spin restriction} and using the description of the normalized Killing forms above, we can compute
\begin{align*}
\phi_{n,m}^\dagger(q_{\Spin_{nm}})&=\phi_{n,m}^\dagger(\frac{1}{2}q_{\SO_{nm}}) = \frac{1}{2}\rho_{n,m}^\dagger(q_{\SO_{nm}})\\
&=\frac{1}{2}(mq_{\SO_{n}},nq_{\SO_{m}}) = (mq_{\Spin_{n}},nq_{\Spin_{m}}).
\end{align*}
Hence the Rost multipliers of $\phi_{n,m}$ are $(m,n)$ also.
\end{proof}

\begin{theo}\label{theo:Rost}
Let $\phi_{\SP}'$ and  $\phi_{2n,2m}'$ be the maps of propositions \ref{prop: PSp induced} and \ref{prop: PSO induced} respectively. Then their Rost multipliers are equal and they are described by
\begin{align*}
(\phi_{\SP}')^\dagger(q_{\HSpin_{4nm}})&=(a q_{\PSP_{2n}}, b q_{\PSP_{2m}}) \\
(\phi_{2n,2m}')^\dagger(q_{\HSpin_{4nm}})&= ( a q_{\PSO_{2n}}, b q_{\PSO_{2m}})
\end{align*}
\[
\text{where } (a,b) =\left\{ \begin{array}{cc}n\pmod{4} & \begin{array}{c|cccc}
& \multicolumn{4}{c}{m \pmod{4}} \\
& 0 & 1 & 2 & 3 \\ \hline
0 & (m,n)&(m,\frac{n}{4}) &(m,\frac{n}{2}) & (m,\frac{n}{4}) \\
1 &(\frac{m}{4},n) & &(\frac{m}{2},n) & \\
2 &(\frac{m}{2},n) &(m,\frac{n}{2}) &(\frac{m}{2},\frac{n}{2}) &(m,\frac{n}{2}) \\
3 &(\frac{m}{4},n) & &(\frac{m}{2},n) & \\
\end{array}
\end{array}\right.
\]
with gaps appearing since these maps only occur when at least one of $n$ and $m$ is even. 
\end{theo}
\begin{proof}
By restricting the commutative diagrams on the right to maximal tori and then dualizing we obtain the adjacent diagrams on the right.
\[
\begin{tikzcd}
\SP_{2n}\times \SP_{2m} \arrow{r}{\phi_{\SP}} \arrow[twoheadrightarrow]{d} & \Spin_{4nm} \arrow[twoheadrightarrow]{d} \\
\PSP_{2n}\times \PSP_{2m} \arrow{r}{\phi_{\SP}'} & \HSpin_{4nm} \\
\end{tikzcd} \quad
\begin{tikzcd}
T^*_{\SP_{2n}}\oplus T^*_{\SP_{2m}} & T^*_{\Spin_{4nm}} \arrow{l}[swap]{\phi_{\SP}^*}\\
T^*_{\PSP_{2n}}\oplus T^*_{\PSP_{2m}}\arrow[hookrightarrow]{u} & T^*_{\HSpin_{4nm}}.\arrow{l}[swap]{(\phi_{\SP}')^*} \arrow[hookrightarrow]{u} \\
\end{tikzcd}
\]
\[
\begin{tikzcd}
\Spin_{2n}\times \Spin_{2m} \arrow{r}{\phi_{\SP}} \arrow[twoheadrightarrow]{d} & \Spin_{4nm} \arrow[twoheadrightarrow]{d} \\
\PSO_{2n}\times \PSO_{2m} \arrow{r}{\phi_{2n,2m}'} & \HSpin_{4nm} \\
\end{tikzcd} \quad
\begin{tikzcd}
T^*_{\Spin_{2n}}\oplus T^*_{\Spin_{2m}} & T^*_{\Spin_{4nm}} \arrow{l}[swap]{\phi_{2n,2m}^*}\\
T^*_{\PSO_{2n}}\oplus T^*_{\PSO_{2m}}\arrow[hookrightarrow]{u} & T^*_{\HSpin_{4nm}}.\arrow{l}[swap]{(\phi_{2n,2m}')^*} \arrow[hookrightarrow]{u} \\
\end{tikzcd}
\]
From the top diagrams we see that $(\phi_{\SP}')^*$ can be viewed as a restriction of $\phi_{\SP}^*$, and therefore $(\phi_{\SP}')^\dagger$ can be viewed as a restriction of $\phi_{\SP}^\dagger$. Similarly, the bottom diagrams show that $(\phi_{2n,2m}')^\dagger$ can be viewed as a restriction of $\phi_{2n,2m}^\dagger$. Then, since we know the Rost multipliers of $\phi_{\SP}$ and $\phi_{2n,2m}$ from propositions \ref{prop:SP multipliers} and \ref{prop:Spin multipliers} respectively we can compute
\begin{align*}
(\phi_{\SP}')^\dagger(q_{\HSpin_{4nm}})&= \phi_{\SP}^\dagger(cq_{\Spin_{4nm}})=(cmq_{\SP_{2n}},cnq_{\SP_{2m}})\\
&=\big(\frac{cm}{d_1}q_{\PSP_{2n}},\frac{cn}{d_2}q_{\PSP_{2m}}\big)
\end{align*}
where $c,d_1,d_2\in\Z$ are integers depending on $n\pmod{4}$ and $m\pmod{4}$ as described at the beginning of section \ref{sec:Rost}. Sorting through the possible cases yields that the Rost multipliers are $(a,b)$ as described in the table above. Similarly we can compute
\begin{align*}
(\phi_{2n,2m}')^\dagger(q_{\HSpin_{4nm}})&= \phi_{2n,2m}^\dagger(cq_{\Spin_{4nm}})=(2cmq_{\Spin_{2n}},2cnq_{\Spin_{2m}})\\
&=\left(\frac{2cm}{2d_1}q_{\PSO_{2n}},\frac{2cn}{2d_2}q_{\PSO_{2m}}\right)
\end{align*}
for the same $c,d_1,d_2\in\Z$ as in the previous case, and therefore we obtain the same Rost multipliers $(a,b)$.
\end{proof}

\begin{theo}
Let $\phi_{2n,2m+1}'$ be the map from proposition \ref{prop:HSpin induced}. Then its Rost multipliers are described by
\[
(\phi_{2n,2m+1}')^\dagger(q_{\HSpin_{2n(2m+1)}}) = (a q_{\HSpin_{2n}}, b q_{\SO_{2m+1}})
\]
\[
\text{where } (a,b)=\begin{cases} (2m+1,n) & \frac{n}{2}\cong 0 \pmod{4} \\ (2m+1,2n) & \frac{n}{2} \cong 2 \pmod{4} \\ (2m+1,4n) & \frac{n}{2} \cong 1,3 \pmod{4} \end{cases}.
\]
\end{theo}
\begin{proof}
Restricting the commutative diagram on the top to maximal tori and dualizing yields the diagram on the bottom
\[
\begin{tikzcd}
\Spin_{2n}\times\Spin_{2m+1} \arrow[twoheadrightarrow]{d} \arrow{r}{\phi_{2n,2m+1}} & \Spin_{2n(2m+1)} \arrow[twoheadrightarrow]{d} \\
\HSpin_{2n}\times\SO_{2m+1} \arrow{r}{\phi_{2n,2m+1}'} & \HSpin_{2n(2m+1)}
\end{tikzcd}
\]
\[
\begin{tikzcd}
\T^*_{\Spin_{2n}}\oplus \T^*_{\Spin_{2m+1}} & \T^*_{\Spin_{2n(2m+1)}} \arrow{l}[swap]{\phi_{2n,2m+1}^*} \\
\T^*_{\HSpin_{2n}}\oplus \T^*_{\SO_{2m+1}} \arrow[hookrightarrow]{u} & \T^*_{\HSpin_{2n(2m+1)}} \arrow{l}[swap]{(\phi_{2n,2m+1}')^*} \arrow[hookrightarrow]{u}
\end{tikzcd}
\]
As in previous proofs this means we can consider $(\phi_{2n,2m+1}')^\dagger$ as a restriction of $\phi_{2n,2m+1}^\dagger$. Then we can compute
\begin{align*}
(\phi_{2n,2m+1}')^\dagger (q_{\HSpin_{2n(2m+1)}})&= \phi_{2n,2m+1}^\dagger(cq_{\Spin_{2n(2m+1)}})\\
&=((2m+1)cq_{\Spin_{2n}},2ncq_{\Spin_{2m+1}})\\
&=\left(\frac{(2m+1)c}{d_1}q_{\HSpin_{2n}},\frac{2nc}{2}q_{\SO_{2m+1}}\right)
\end{align*}
where $c,d_1\in\Z$ depend on $n/2\pmod{4}$. Sorting through the possible cases yields that the Rost multipliers of $\phi_{2n,2m+1}'$ are the $(a,b)$ given above.
\end{proof}

We briefly note that the results given above agree with the Rost multipliers appearing in \cite{Gar} on the case in which they overlap. For the map $\varphi \colon \SP_2\times \SP_8 \rightarrow \Spin_{16}$ it is noted in \cite[(7.2)]{Gar} that the restriction
$\varphi|_{\SP_8} \colon \SP_8 \rightarrow \Spin_{16}$
has a Rost multiplier of $1$. Comparing this to our map $\phi \colon \SP_{2n}\times \SP_{2m} \rightarrow \Spin_{4nm}$ for the case $n=1,m=4$ we see that the duals of $\phi$ and $\phi|_{\SP_8}$ are given by
\begin{align*}
\phi^* \colon T^*_{\Spin_{16}} &\rightarrow T^*_{\SP_{2(1)}} \oplus T^*_{\SP_{2(4)}} & \phi|_{\SP_8}^* \colon T^*_{\Spin_{16}} &\rightarrow T^*_{\SP_8} \\
q_{\Spin_{16}} &\mapsto \left( 4\cdot q_{\SP_2},1\cdot q_{\SP_8}\right) &  q_{\Spin_{16}} &\mapsto 1\cdot q_{\SP_8}
\end{align*}
and hence $\phi|_{\SP_8}$ also has a Rost multiplier of 1.

\section{Application to Degree 3 Cohomological Invariants}\label{sec:cohomology}
Following \cite{Mer} for a linear algebraic group $G$ over a field $\field$ we consider its group of degree 3 cohomological invariants with coefficients in $\Q/\Z(2)$, denoted $\Inv^3(G,2)$, where such an invariant is a natural transformation
\[
\alpha \colon H^1(-,G) \to H^3(-,\Q/\Z(2)).
\]
of the functors $H^1(-,G), H^3(-, \Q/\Z(2)) \colon \text{Fields}/\field \to \text{Groups}$. Such an invariant is called \emph{normalized} if it maps the trivial $G$-torsor to $0$, and the group of normalized invariants is denoted $\Inv^3(G,2)_{\text{norm}}$. Note that 
\[
\Inv^3(G,2) = H^3(\field,\Q/\Z(2)) \oplus \Inv^3(G,2)_{\text{norm}}.
\] 
Forming a subgroup of the normalized invariants are the \emph{decomposable} invariants, $\Inv^3(G,2)_{\text{dec}}$, which are those $\alpha \in \Inv^3(G,2)_{\text{norm}}$ given by a finite sum of cup products of $\phi_i \in \field^\times$ with invariants $\alpha'_i \colon H^1(-,G) \to H^2(-,\Q/\Z(1))$ of degree 2. That is, for all $\Efield$ a field extension of $\field$ and $Y\in H^1(\Efield, G)$,
\[
\alpha(Y) = \sum \phi_i \cup \alpha'_i(Y).
\]
For $\PSP_{2n}$ the decomposable invariants are described in \cite[Theorem 4.6]{Mer} and are
\[
\Inv^3(\PSP_{2n},2)_{\text{dec}} \cong \field^\times/(\field^\times)^2
\]
where the isomorphisms is as follows. For $(c)\in\field^\times/(\field^\times)^2 \cong \Inv^3(\PSP_{2n},2)_{\text{dec}}$ it behaves over a field extension $\Efield/\field$ as
\begin{align*}
(c)\colon H^1(\Efield,\PSP_{2n})&\to H^3(\Efield,\Q/\Z(2)) \\
[(A,\psi)]&\mapsto (c)\cup[A]
\end{align*}
where $(c)\cup[A]$ is a cohomological cup product considered an element of $H^3(\Efield,\Q/\Z(2))$ by the identifications
\begin{align*}
(c)\in\Efield^\times/(\Efield^\times)^2 &\cong H^1(\Efield,\mu_2)\leq H^1(\Efield,\Q/\Z) \\
[A]\in\Br(\Efield) &\cong H^2(\Efield,\Q/\Z) \\
H^1(\Efield,\Q/\Z)\cup H^2(\Efield,\Q/\Z) &\to H^3(\Efield,\Q/\Z(2)).
\end{align*}
Similarly, we have that the decomposable invariants of $\HSpin_{4n}$ are
\[
\Inv^3(\HSpin_{4n},2)_{\text{dec}} \cong \field^\times/(\field^\times)^2.
\]
For $(c)\in\field^\times/(\field^\times)^2 \cong\Inv^3(\HSpin_{4n},2)_{\text{dec}}$ it behaves over $\Efield$ as
\begin{align*}
(c)\colon H^1(\Efield,\HSpin_{4n})&\to H^3(\Efield,\Q/\Z(2))\\
Y &\mapsto (c)\cup[A(Y)]
\end{align*}
where $[A(Y)]$ is the class of the algebra represented by the image of $Y$ under the map $H^1(\Efield,\HSpin_{4n})\to H^1(\Efield,\PSO_{4n})$ induced by the natural projection. This can be seen from a remark in \cite[\S 1a]{Mer} which tells us that
\[
\Inv^2(\HSpin_{4n},1)_{\text{norm}} \cong \mu_2^* \cong \Z/2\Z
\]
since $\mu_2$ is the kernel of the universal covering $\Spin_{4n}\twoheadrightarrow \HSpin_{4n}$. Therefore the degree two invariant $Y\mapsto [A(Y)]$, which comes from a degree two invariant of $PSO_{4n}$, is the non-trivial degree two invariant of $\HSpin_{4n}$, from which the statement follows.

Additionally, we consider the quotient of the normalized invariants by the decomposable invariants, denoted
\[
\Inv^3(G,2)_{\text{ind}} := \frac{\Inv^3(G,2)_{\text{norm}}}{\Inv^3(G,2)_{\text{dec}}}
\]
and called \emph{indecomposable} invariants. We take from \cite{Mer} and \cite{BR} some descriptions of these groups of indecomposable invariants.
\begin{align*}
\Inv^3(\PSP_{2n},2)_{\text{ind}} &\cong \begin{cases} \Z/2\Z & n \equiv 0 \pmod{4} \\ 0 & \text{else} \end{cases} \\
\Inv^3(\PSO_{2n},2)_{\text{ind}} &\cong \begin{cases} \Z/2\Z & n \equiv 0 \pmod{4} \\ 0 & \text{else} \end{cases} \\
\Inv^3(\HSpin_{4n},2)_{\text{ind}} &\cong \begin{cases} 0 & n>1 \text{ odd or } n=2 \\ \Z/2\Z & n\equiv 2\pmod{4}, n\neq 2 \\ \Z/4\Z & n \equiv 0 \pmod{4} \end{cases}
\end{align*}

As noted in \cite{BR}, the above groups of invariants are functorial in $G$, that is a homomorphism of linear algebraic groups $\rho \colon G\to H$ induces a commutative diagram
\[
\begin{tikzcd}
\Inv^3(H,2)_{\text{dec}} \arrow{r} \arrow{d} & \Inv^3(H,2)_{\text{norm}} \arrow{r} \arrow{d} & \Inv^3(H,2)_{\text{ind}} \arrow{d} \\
\Inv^3(G,2)_{\text{dec}} \arrow{r} & \Inv^3(G,2)_{\text{norm}} \arrow{r} & \Inv^3(G,2)_{\text{ind}} \\
\end{tikzcd}
\]
where the vertical maps send an invariant $\Delta$ to the composition
\[
H^1(-,G)\to H^1(-,H)\overset{\Delta}{\to} H^3(-,\Q/\Z).
\]
Now we can use our map $\phi_{\SP}'$ of proposition \ref{prop: PSp induced} within this framework to describe the structure of $\Inv^3(\HSpin_{4n},2)_{\text{norm}}$.

\begin{lem}\label{lem:decomposable pullback}
Let $\phi_{\SP}'$ be the map of proposition \ref{prop: PSp induced}. Then its induced map on decomposable invariants is the diagonal map
\begin{align*}
\field^\times/(\field^\times)^2 &\to \field^\times/(\field^\times)^2\oplus \field^\times/(\field^\times)^2 \\
(c)&\mapsto \big((c),(c)\big)
\end{align*}
which in particular is injective.
\end{lem}
\begin{proof}
Let $\Efield/\field$ be a field extension. We consider the following commutative diagram of cohomology sets where the horizontal map is induced by $\phi_{\SP}'$, the vertical map is induced by the natural projection, and the diagonal map is induced by the Kronecker tensor product map between adjoint groups.
\[
\begin{tikzcd}
H^1(\Efield,\PSP_{2n}\times\PSP_{2m}) \arrow{r}\arrow{dr} & H^1(\Efield,\HSpin_{4nm}) \arrow{d} \\
 & H^1(\Efield,\PSO_{4nm})
\end{tikzcd}
\]
Since the diagonal map comes from the tensor product map, if we consider central simple algebras $(A_1,\psi_1),(A_2,\psi_2)$ over $\Efield$ of degree $2n$ and $2m$ respectively with symplectic involution then we can denote their images within the above diagram as
\[
\begin{tikzcd}
([A_1,\psi_1],[A_2,\psi_2])\arrow[mapsto]{r}\arrow[mapsto]{dr} & Y \arrow[mapsto]{d} \\
 & \left[A_1\otimes A_2,\psi_1 \otimes \psi_2\right].\\
\end{tikzcd}
\]
Now for a decomposable degree three invariant of $\HSpin$, $\Delta \in \Inv^3(\HSpin_{4nm},2)_{\text{dec}}$, $\Delta$ corresponds to some $(c)\in \field^\times/(\field^\times)^2$ and we have that
\[
\Delta(Y) = (c)\cup[A(Y)] = (c)\cup[A_1\otimes A_2].
\]
Therefore, denoting the image of $\Delta$ under the map 
\[
\Inv^3(\HSpin_{4nm},2)_{\text{dec}} \to \Inv^3(\PSP_{2n}\times\PSP_{2m},2)_{\text{dec}}
\]
by $\Delta'$ we have that
\[
\Delta'([A_1,\psi_1],[A_2,\psi_2])=\Delta(Y)=(c)\cup[A_1\otimes A_2] = (c)\cup[A_1] + (c)\cup[A_2]
\]
since addition in $\Br(\Efield)$ is the tensor product of algebras. Hence we see that $\Delta'$ corresponds to the diagonal element $\big((c),(c)\big) \in \field^\times/(\field^\times)^2\oplus \field^\times/(\field^\times)^2$.
\end{proof}

\begin{theo}\label{theo:invariant splitting}
Let $n\geq 2$. The exact sequence 
\[
0 \to \Inv^3(\HSpin_{4n},2)_{\emph{dec}} \to \Inv^3(\HSpin_{4n},2)_{\emph{norm}} \to \Inv^3(\HSpin_{4n},2)_{\emph{ind}} \to 0
\]
is split. In particular, we can describe the degree three normalized invariants of $\HSpin_{4n}$ as
\[
\Inv^3(\HSpin_{4n},2)_{\emph{norm}} \cong \begin{cases}
\field^\times/(\field^\times)^2 & n \text{ is odd or } n=2 \\
\field^\times/(\field^\times)^2 \oplus \Z/2\Z & n\equiv 2\pmod{4} \text{ and } n\neq 2 \\
\field^\times/(\field^\times)^2 \oplus \Z/4\Z & n\equiv 0\pmod{4}.
\end{cases}
\]
\end{theo}
\begin{proof}
First, we note that if $n>1$ is odd or $n=2$ then $\Inv^3(\HSpin_{4n},2)_{\text{ind}}\cong 0$ and the result is immediate. 

Next assume that $n\equiv 2\pmod{4}$ and $n\neq 2$. In this case we can write $n=2m$ where $m\geq 3$ is an odd integer. Then we consider the commutative diagram of degree three invariants induced by the map $\phi_{\SP}'\colon \PSP_{4}\times \PSP_{2m} \to \HSpin_{4n}$. In this case it takes the form
\[
\begin{tikzcd}
\field^\times/(\field^\times)^2 \arrow[hookrightarrow]{r} \arrow[hookrightarrow]{d} & \Inv^3(\HSpin_{4n},2)_{\text{norm}} \arrow[twoheadrightarrow]{r} \arrow{d} & \Z/2\Z \arrow{d} \\
\field^\times/(\field^\times)^2 \oplus \field^\times/(\field^\times)^2 \arrow[hookrightarrow]{r} & \field^\times/(\field^\times)^2 \oplus \field^\times/(\field^\times)^2 \arrow{r} & 0
\end{tikzcd}
\]
where the leftmost vertical map is the injection of lemma \ref{lem:decomposable pullback}. From this diagram we see that for all $\Delta \in \Inv^3(\HSpin_{4n},2)_{\text{norm}}$, $2\Delta$ maps to $0$ in both $\Z/2\Z$ and $\field^\times/(\field^\times)^2 \oplus \field^\times/(\field^\times)^2$. Therefore, $2\Delta$ comes from $\field^\times/(\field^\times)^2$ and maps to $0$ along an injection, meaning $2\Delta=0$. Thus, the exponent of $\Inv^3(\HSpin_{4n},2)_{\text{norm}}$ is 2, the same as of $\Z/2\Z$, and so any choice of preimage of $1\in \Z/2\Z$ produces a splitting as desired.

Finally, assume that $n\equiv 0\pmod{4}$. Again we write $n=2m$ where now $m$ is an even integer, and we consider the commutative diagram of degree three invariants induced by the map $\phi_{\SP}'\colon \PSP_{4}\times\PSP_{2m} \to \HSpin_{4n}$. In this case, the group $\Inv^3(\PSP_4\times\PSP_{2m},2)_{\text{ind}}$ is either $0$ or $\Z/2\Z$ depending on $m\pmod{4}$. Therefore the commutative diagram takes the form
\[
\begin{tikzcd}
\field^\times/(\field^\times)^2 \arrow[hookrightarrow]{r} \arrow[hookrightarrow]{d} & \Inv^3(\HSpin_{4n},2)_{\text{norm}} \arrow[twoheadrightarrow]{r} \arrow{d} & \Z/4\Z \arrow{d} \\
\field^\times/(\field^\times)^2 \oplus \field^\times/(\field^\times)^2 \arrow[hookrightarrow]{r} & \Inv^3(\PSP_4\times \PSP_{2m},2)_{\text{norm}} \arrow{r} & 0 \text{ or } \Z/2\Z
\end{tikzcd}
\]
Since the bottom row is exact, the exponent of $\Inv^3(\PSP_4\times \PSP_{2m},2)_{\text{norm}}$ divides $2$ if $\Inv^3(\PSP_4\times \PSP_{2m},2)_{\text{ind}}\cong 0$, or it divides $4$ if $\Inv^3(\PSP_4\times \PSP_{2m},2)_{\text{ind}}\cong \Z/2\Z$. In either case the exponent divides $4$. Hence, for all $\Delta \in \Inv^3(\HSpin_{4n},2)_{\text{norm}}$, $4\Delta$ maps to 0 in both $\Z/4\Z$ and $\Inv^3(\PSP_4\times \PSP_{2m},2)_{\text{norm}}$. Therefore, $4\Delta$ comes from $\field^\times/(\field^\times)^2$ and maps to 0 along an injection, meaning $4\Delta=0$. Thus, the exponent of $\Inv^3(\HSpin_{4n},2)_{\text{norm}}$ is $4$, the same as $\Z/4\Z$, and so any choice of preimage of $1\in \Z/4\Z$ produces a splitting as desired. 
\end{proof}
We note that theorem \ref{theo:invariant splitting} is a generalization of \cite[Corollary 5.2]{BR} which states the result for $\HSpin_{16}$.

\subsection{An Explicit Invariant of $\HSpin$}\label{sec:invariant}
In \cite{Mer}, Merkurjev constructed a non-trivial, non-decomposable invariant $I\in \Inv^3(\PSO_{2n},2)_{\text{norm}}$ when $n\equiv 0\pmod{4}$ and the characteristic of $\field$ is different from 2. $I$ is therefore a representative of $1\in \Z/2\Z \cong \Inv^3(\PSO_{2n},2)_{\text{ind}}$. We briefly summarize the construction. Consider an element $(A,\tau,e)\in H^1(\field,\PSO_{2n})$ where $(A,\tau)$ is a central simple $\field$-algebra with orthogonal involution of trivial discriminant, and $e\in Z(C(A,\tau))$ is a non-trivial idempotent from the center of the Clifford algebra. Merkurjev finds a hyperbolic involution $\sigma'$ on $A$ and another corresponding idempotent $e'$ from the center of $C(A,\sigma')$. Merkurjev then notes that $(A,\sigma',e')$ lies in the image of the natural map
\[
H^1(\field,\SO(A,\sigma)) \to H^1(\field,\PSO(A,\sigma)).
\]
The set $H^1(\field,\SO(A,\sigma))$ is described as equivalence classes of pairs $(a,x)\in A\times \field$ where $a=\sigma(a)$ and $x^2=\text{Nrd}(a)$. Since $(A,\sigma',e')$ is in the image of this map, it comes from some $(a,x)$. The invariant then acts by
\[
I(A,\sigma,e)=(x)\cup[A]\in H^3(\field,\Q/\Z(2)).
\]
Now, let $n\equiv 0\pmod{4}$ and $n\neq 4$. We consider the above invariant pulled back to $\HSpin_{2n}$,
\begin{align*}
\Inv^3(\PSO_{2n},2)_{\text{norm}} &\to \Inv^3(\HSpin_{2n},2)_{\text{norm}} \\
I &\mapsto I'
\end{align*}
and we see that for $Y\in H^1(\field,\HSpin_{2n})$ it acts via
\[
I'(Y)=(x)\cup [A(Y)]
\]
where $[A(Y)]$ is the algebra represented by the image of $Y$ in $H^1(\field,\PSO_{2n})$, and $x\in \field$ is chosen through the process outlined above. As a result of \cite[Remark 3.10]{Mer}, since the map $\pi\colon\HSpin_{2n}\twoheadrightarrow \PSO_{2n}$ is between split groups, the induced map
\[
\Inv^3(\PSO_{2n},2)_{\text{ind}} \to \Inv^3(\HSpin_{2n},2)_{\text{ind}}
\]
is described by the Rost multiplier of the homomorphism. In this case the dual map $\pi^*\colon \T^*_{\PSO_{2n}}\hookrightarrow \T^*_{\HSpin_{2n}}$ is simply the inclusion as a subgroup, and so by referencing the normalized Killing forms given in section \ref{sec:Rost} we see the Rost multiplier is $1$ if $n/2 \equiv 2\pmod{4}$, and is $2$ if $n/2 \equiv 0\pmod{4}$. Therefore the induced map between indecomposable invariants is
\begin{align*}
\Z/2\Z &\overset{\sim}{\to} \Z/2\Z & \text{if }n/2 \equiv 2\pmod{4} \\
\Z/2\Z &\hookrightarrow \Z/4\Z &\text{if } n/2 \equiv 0\pmod{4} \\
1 &\mapsto 2 &
\end{align*}
\begin{rem}\label{explicit invariant}
In the case when $m\equiv 2\pmod{4}$, the invariant $I'$ above, along with the decomposable invariants, gives a explicit description of all degree three normalized invariants $\Inv^3(\HSpin_{4m},2)_{\text{norm}}$.
\end{rem}

\section{Appendix}
\renewcommand\thetheo{\arabic{theo}}
\setcounter{theo}{0}
The simply connected groups $\Spin_n$ and $\SP_{2n}$ are defined by the relations given in \cite[Theorem 8]{Stein}\label{less relations}. The constants $c_{ij}$ appear below. The Galois action given for each group is the entry-wise action on matrix entries in the case of $\SP_{2n}$ and $\SO_d$. The action is defined analogously for $\Spin_d$ making the natural projection $\Spin_d\twoheadrightarrow \SO_d$ a $\Galgroup$-morphism.
\begin{app}\label{SP app} $\SP_{2n}$ is of type $C_n$ and by \cite{Bou4-6} it has root system 
\[
\Phi = \{\pm e_i \pm e_j,\pm2e_k \mid 1\leq i < j \leq n, 1\leq k \leq n\}.
\]
The $\fieldsep$-points of $\SP_{2n}$ have Chevalley generators
\begin{align*}
x_{e_i-e_j}(t) &= I_{2n}+ t(E_{i,j}-E_{2n+1-j,2n+1-i})  \\
x_{e_i+e_j}(t) &=I_{2n}+t(E_{i,2n+1-j}+E_{j,2n+1-i}) \\
x_{2e_j}(t) &= I_{2n}+tE_{j,2n+1-j}
\end{align*}
for $t\in \fieldsep$, $1\leq i \leq n-1, 1\leq j \leq n$ with $i<j$, and $x_{-\alpha}(t)=x_\alpha(t)^T$. For all $\alpha \in \Phi$ and for all $\sigma \in \Galgroup=\Gal(\fieldsep/\field)$, we have 
\[
\sigma(x_\alpha(t))=x_\alpha(\sigma(t)).
\]
For $\alpha,\beta \in \Phi$, if $0\neq \alpha+\beta \notin \Phi$ then $(x_\alpha(t),x_\beta(u))$ is trivial. Otherwise, if $\alpha+\beta \in \Phi$ then it falls into one of the following cases. For integers $i,j,k,l \in [1,n]$ with $i<j,k,l$ and $k<l$, and for $a_1,a_2,a_3,a_4 \in \{ 1, -1\}$ the commutators are
\begin{align*}
j=k:& (x_{a_1e_i + a_2e_j}(t), x_{-a_2e_j + a_3e_l}(u)) = x_{a_1e_i + a_3e_l}(c tu) \\
i\neq k, j=l:& (x_{a_1e_i + a_2e_j}(t), x_{a_3e_k - a_2e_j}(u))= x_{a_1e_i + a_3e_k}(ctu) \\
i=k, j\neq l:& (x_{a_2e_i + a_1e_j}(t), x_{-a_2e_i + a_3e_l}(u)) = \begin{cases}x_{a_1e_j + a_3e_l}(ctu) & j<l \\ x_{a_3e_l + a_1e_j}(ctu) & l< j. \end{cases}
\end{align*}
where $c=a_2 \cdot \min\{a_1a_2, -a_2a_3\}$. Furthermore,
\begin{align*}
(x_{a_1e_i + a_2e_j}(t), x_{a_1e_i - a_2e_j}(u)) &= x_{2a_1e_i}(-2a_2 tu) \\
(x_{a_1e_i + a_2e_j}(t), x_{-a_1e_i + a_2e_j}(u)) &= x_{2a_2e_j}(-2a_1 tu) \\
(x_{a_1e_i + a_2e_j}(t), x_{-2a_1e_i}(u)) &= x_{-a_1e_i+a_2e_j}(a_2 tu) \cdot x_{2a_2e_j}(-a_1a_2 t^2 u) \\
(x_{a_1e_i + a_2e_j}(t), x_{-2a_2e_j}(u)) &= x_{a_1e_i - a_2e_j}(a_1 tu) \cdot x_{2a_1e_i}(-a_1a_2 t^2 u) \\
\end{align*}
\end{app}

\begin{app}\label{Spin app} $\SO_d$ is of type $D_n$ if $d=2n$, or it is of type $B_n$ if $d=2n+1$.Therefore by \cite{Bou4-6} it has root system $\Phi = \{\pm e_i \pm e_j \mid 1\leq i < j \leq n\}$ or $\Phi = \{\pm e_i \pm e_j,\pm e_k \mid 1\leq i < j \leq n, 1\leq k \leq n\}$ respectively. Letting $\overline{i}=d+1-i$, the $\fieldsep$-points of $\SO_d$ have Chevalley generators
\begin{align*}
x_{e_i-e_j}(t) &=I_d + t(E_{i,j}-E_{\overline{j},\overline{i}}) &  x_{e_i+e_j}(t) &=I_d+ t(E_{i,\overline{j}}-E_{j,\overline{i}})
\end{align*}
in all cases, and when $d=2n+1$ they have additional generators
\begin{align*}
x_{e_j}(t) &= I_d + \sqrt{2}t E_{j,n+1} - \sqrt{2}t E_{n+1,\overline{j}} - t^2 E_{j,\overline{j}}
\end{align*}
for $t\in \Efield$, $1\leq i \leq n-1, 1\leq j \leq n$ with $i<j$, and $x_{-\alpha}(t)=x_\alpha(t)^T$. For $\sigma\in \Galgroup=\Gal(\fieldsep/\field)$ and for $\alpha\in\Phi$ a long root, i.e. $\alpha=\pm e_i \pm e_j$, we have $\sigma(x_\alpha(t))=x_\alpha(\sigma(t))$. If $\alpha\in\Phi$ is a short root, i.e. $\alpha=\pm e_i$, then
\[
\sigma(x_\alpha(t))=\begin{cases}
x_\alpha(\sigma(t)) & \text{ if } \sigma(\sqrt{2})=\sqrt{2} \\
x_\alpha(-\sigma(t)) & \text{ if } \sigma(\sqrt{2})=-\sqrt{2} \\
\end{cases}
\] 
The commutator relations for $\SO_d$ and $\Spin_d$ are the same. For $\alpha,\beta \in \Phi$, if $0\neq \alpha+\beta \notin \Phi$ then $(x_\alpha(t),x_\beta(u))$ is trivial. Otherwise, if $\alpha+\beta \in \Phi$ then it falls into one of the following cases.  For integers $i,j,k,l \in [1,n]$ with $i<j,k,l$ and $k<l$, and for $a_1,a_2,a_3,a_4 \in \{ 1, -1\}$ the commutators are
\begin{align*}
j=k:& (x_{a_1e_i + a_2e_j}(t), x_{-a_2e_j + a_3e_l}(u)) = x_{a_1e_i + a_3e_l}(-a_2 tu) \\
i\neq k, j=l:& (x_{a_1e_i + a_2e_j}(t), x_{a_3e_k - a_2e_j}(u))= x_{a_1e_i + a_3e_k}(-a_3 tu) \\
i=k, j\neq l:& (x_{a_1e_i + a_2e_j}(t), x_{-a_1e_i + a_3e_l}(u)) = \begin{cases}x_{a_2e_j + a_3e_l}(a_2 tu) & j<l \\ x_{a_3e_l + a_2e_j}(-a_3 tu) & l< j. \end{cases}
\end{align*}
and furthermore if $d$ is odd there are additional relations
\begin{align*}
(x_{a_1e_i + a_2e_j}(t),x_{-a_1e_i}(u))&= x_{-a_1e_i + a_2e_j}(-tu^2) x_{a_2e_j}(a_2 tu) \\
(x_{a_1e_i + a_2e_j}(t),x_{-a_2e_j}(u))&= x_{a_1e_i - a_2e_j}(tu^2) x_{a_1e_i}(-a_2 tu)\\
(x_{a_1e_i}(t),x_{a_2e_j}(u)) &= x_{a_1e_i + a_2e_j}(-2a_2 tu).
\end{align*}
\end{app}
For all groups, commutators where $k$ is the minimal index involved instead of $i$ are described by taking the inverse of the appropriate relation above. The constants for $\SP_{2n}$ were calculated using the natural representation $\spLie_{2n} \hookrightarrow \Mat_{2n}(\field)$. The constants for $\Spin_d$ were calculated within $\SO_d$ using the natural representation $\so_d\hookrightarrow \Mat_d(\field)$. By \cite[Lemma 15.1]{Stein} these constants do not depend on the chosen representation of $\so_d$ and so also apply to $\Spin_d$. 

\begin{app}\label{Images}
The images of the map $\rho_{\SP}$ on $\fieldsep$-points in terms of Chevalley generators.
\[
\begin{array}{l|l}
x\in \SP_{2n}\times \SP_{2m} & \rho_{\SP}(x) \in \SO_{4nm} \quad \text{where } [y]=(2m)y\\ \hline
(x_{e_i - e_j}(t),I) & \prod_{k=1}^{2m} x_{e_{[i-1]+k}-e_{[j-1]+k}}(t) \\
(x_{-e_i + e_j}(t),I) & \prod_{k=1}^{2m} x_{-e_{[i-1]+k}+e_{[j-1]+k}}(t) \\
(x_{e_i + e_j}(t),I) & \prod_{k=1}^{m}x_{e_{[i-1]+k}+e_{[j-1]+m+k}}(t) x_{e_{[i-1]+m+k}+e_{[j-1]+k}}(-t) \\
(x_{-e_i - e_j}(t),I) & \prod_{k=1}^{m}x_{-e_{[i-1]+k}-e_{[j-1]+m+k}}(t) x_{-e_{[i-1]+m+k}-e_{[j-1]+k}}(-t) \\
(x_{2e_i}(t),I) & \prod_{k=1}^{m} x_{e_{[i-1]+k}+e_{[i-1]+m+k}}(t) \\
(x_{-2e_i}(t),I) & \prod_{k=1}^{m} x_{-e_{[i-1]+k}-e_{[i-1]+m+k}}(t) \\ \hline
(I,x_{e_i - e_j}(u)) & \prod_{k=0}^{n-1} x_{e_{([k]+i}-e_{[k]+j}}(u) x_{-e_{[k]+m+i}+e_{[k]+m+j}}(-u)  \\
(I,x_{-e_i + e_j}(u)) &\prod_{k=0}^{n-1} x_{-e_{[k]+i}+e_{[k]+j}}(u) x_{e_{[k]+m+i}-e_{[k]+m+j}}(-u)\\
(I,x_{e_i + e_j}(u)) &  \prod_{k=0}^{n-1}x_{e_{[k]+i}-e_{[k]+m+j}}(-u) x_{e_{[k]+j}-e_{[k]+m+i}}(-u) \\
(I,x_{-e_i - e_j}(u)) &  \prod_{k=0}^{n-1}x_{-e_{[k]+i}+e_{[k]+m+j}}(-u) x_{-e_{[k]+j}+e_{[k]+m+i}}(-u)\\
(I,x_{2e_i}(u)) & \prod_{k=0}^{n-1} x_{e_{[k]+i}-e_{[k]+m+i}}(-u) \\
(I,x_{-2e_i}(u)) & \prod_{k=0}^{n-1} x_{-e_{[k]+i}+e_{[k]+m+i}}(-u) \\
\end{array}
\]
Images of the maps $\rho_{\SO}$ on $\fieldsep$-points in terms of Chevalley generators.
\[
\begin{array}{l|l}
\SO_{2n}\times \SO_{2m} & \rho_{\SO}(x)\in\SO_{4nm}.\; [y]=(2m)y,\; \overline{y}=2m+1-y\\ \hline
(x_{e_i - e_j}(t),I) & \prod_{k=1}^{2m} x_{e_{[i-1]+k}-e_{[j-1]+k}}(t) \\
(x_{-e_i + e_j}(t),I) & \prod_{k=1}^{2m} x_{-e_{[i-1]+k}+e_{[j-1]+k}}(t)\\
(x_{e_i + e_j}(t),I) & \prod_{k=1}^{2m} x_{e_{[i-1]+k}+e_{[j-1]+\overline{k}}}(t)\\
(x_{-e_i - e_j}(t),I) & \prod_{k=1}^{2m} x_{-e_{[i-1]+k}-e_{[j-1]+\overline{k}}}(t)\\ \hline
(I,x_{e_i - e_j}(u)) & \prod_{k=0}^{n-1} x_{e_{[k]+i}-e_{[k]+j}}(u)x_{e_{[k]+\overline{j}}-e_{[k]+\overline{i}}}(-u)\\
(I,x_{-e_i + e_j}(u)) &\prod_{k=0}^{n-1} x_{-e_{[k]+i}+e_{[k]+j}}(u)x_{-e_{[k]+\overline{j}}+e_{[k]+\overline{i}}}(-u)\\
(I,x_{e_i + e_j}(u)) &  \prod_{k=0}^{n-1}x_{e_{[k]+i}-e_{[k]+\overline{j}}}(u)x_{e_{[k]+j}-e_{[k]+\overline{i}}}(-u) \\
(I,x_{-e_i - e_j}(u)) &  \prod_{k=0}^{n-1} x_{-e_{[k]+i}+e_{[k]+\overline{j}}}(u)x_{-e_{[k]+j}+e_{[k]+\overline{i}}}(-u)\\
\end{array}
\]
\[
\begin{array}{l|l}
\SO_{2n}\times \SO_{2m+1} & \rho_{\SO}(x) \in \SO_{4nm+2n}.\; [y]=(2m+1)y,\; \overline{y}=2m+2-y \\ \hline
(x_{e_i - e_j}(t),I) & \prod_{k=1}^{2m+1} x_{e_{[i-1]+k}-e_{[j-1]+k}}(t) \\
(x_{-e_i + e_j}(t),I) & \prod_{k=1}^{2m+1} x_{-e_{[i-1]+k}+e_{[j-1]+k}}(t) \\
(x_{e_i + e_j}(t),I) &  \prod_{k=1}^{2m+1} x_{e_{[i-1]+k}+e_{[j-1]+\overline{k}}}(t)\\
(x_{-e_i - e_j}(t),I) &  \prod_{k=1}^{2m+1} x_{-e_{[i-1]+k}-e_{[j-1]+\overline{k}}}(t)\\ \hline
(I,x_{e_i - e_j}(u)) & \prod_{k=0}^{n-1}x_{e_{[k]+i}-e_{[k]+j}}(u)x_{e_{[k]+\overline{j}}-e_{[k]+\overline{i}}}(-u) \\
(I,x_{-e_i + e_j}(u)) &\prod_{k=0}^{n-1}x_{-e_{[k]+i}+e_{[k]+j}}(u)x_{-e_{[k]+\overline{j}}+e_{[k]+\overline{i}}}(-u)\\
(I,x_{e_i + e_j}(u)) & \prod_{k=0}^{n-1} x_{e_{[k]+i}-e_{[k]+\overline{j}}}(u)x_{e_{[k]+j}-e_{[k]+\overline{i}}}(-u) \\
(I,x_{-e_i - e_j}(u)) &\prod_{k=0}^{n-1} x_{-e_{[k]+i}+e_{[k]+\overline{j}}}(u)x_{-e_{[k]+j}+e_{[k]+\overline{i}}}(-u) \\
(I,x_{e_i}(u)) & \begin{array}{@{}r@{}l@{}}\prod_{k=0}^{n-1}&\big(x_{e_{[k]+i}-e_{[k]+(m+1)}}(\sqrt{2}u) x_{e_{[k]+(m+1)}-e_{[k]+\overline{i}}}(-\sqrt{2}u)\\ &\cdot x_{e_{[k]+i}-e_{[k]+\overline{i}}}(u^2)\big)\end{array}\\
(I,x_{-e_i}(u)) & \begin{array}{@{}r@{}l@{}}\prod_{k=0}^{n-1} & \big( x_{-e_{[k]+(m+1)}+e_{[k]+\overline{i}}}(-\sqrt{2}u) x_{-e_{[k]+i}+e_{[k]+(m+1)}}(\sqrt{2}u) \\
&\cdot x_{-e_{[k]+i}+e_{[k]+\overline{i}}}(u^2)\big) \end{array}\\
\end{array}
\]
\[
\begin{array}{l|l}
\SO_{2n+1}\times \SO_{2m+1} & \rho_{\SO}(x) \in \SO_{(2n+1)(2m+1)}.\; [y]=(2m+1)y,\; \overline{y}=2m+2-y \\ \hline
(x_{e_i - e_j}(t),I) & \prod_{k=1}^{2m+1} x_{e_{[i-1]+k}-e_{[j-1]+k}}(t) \\
(x_{-e_i + e_j}(t),I) & \prod_{k=1}^{2m+1} x_{-e_{[i-1]+k}+e_{[j-1]+k}}(t) \\
(x_{e_i + e_j}(t),I) &  \prod_{k=1}^{2m+1} x_{e_{[i-1]+k}+e_{[j-1]+\overline{k}}}(t)\\
(x_{-e_i - e_j}(t),I) &  \prod_{k=1}^{2m+1} x_{-e_{[i-1]+k}-e_{[j-1]+\overline{k}}}(t)\\
(x_{e_i}(t),I) & \begin{array}{@{}r@{}l@{}} \big(\prod_{k=1}^{m}& x_{e_{[i-1]+k}-e_{[n]+k}}(\sqrt{2}t)x_{e_{[i-1]+\overline{k}}+e_{[n]+k}}(\sqrt{2}t) \\
& \cdot x_{e_{[i-1]+k}+e_{[i-1]+\overline{k}}}(t^2)\big) \cdot x_{e_{[i-1]+(m+1)}}(t) \end{array} \\
(x_{-e_i}(t),I) & \begin{array}{@{}r@{}l@{}} \big(\prod_{k=1}^{m} & x_{-e_{[i-1]+\overline{k}}-e_{[n]+k}}(\sqrt{2}t)x_{-e_{[i-1]+k}+e_{[n]+k}}(\sqrt{2}t) \\
&\cdot x_{-e_{[i-1]+k}-e_{[i-1]+\overline{k}}}(t^2)\big)\cdot x_{-e_{[i-1]+(m+1)}}(t) \end{array}\\ \hline
(I,x_{e_i - e_j}(u)) & \prod_{k=0}^{n-1} \big(x_{e_{[k]+i}-e_{[k]+j}}(u)x_{e_{[k]+\overline{j}}-e_{[k]+\overline{i}}}(-u)\big) \cdot x_{e_{[n]+i}-e_{[n]+j}}(u) \\
(I,x_{-e_i + e_j}(u)) &\prod_{k=0}^{n-1} \big( x_{-e_{[k]+i}+e_{[k]+j}}(u)x_{-e_{[k]+\overline{j}}+e_{[k]+\overline{i}}}(-u)\big) \cdot x_{-e_{[n]+i}+e_{[n]+j}}(u)\\
(I,x_{e_i + e_j}(u)) &\prod_{k=0}^{n-1}\big( x_{e_{[k]+i}-e_{[k]+\overline{j}}}(u)x_{e_{[k]+j}-e_{[k]+\overline{i}}}(-u)\big)\cdot  x_{e_{[n]+i}+e_{[n]+j}}(u) \\
(I,x_{-e_i - e_j}(u)) &\prod_{k=0}^{n-1} \big( x_{-e_{[k]+i}+e_{[k]+\overline{j}}}(u)x_{-e_{[k]+j}+e_{[k]+\overline{i}}}(-u)\big)\cdot  x_{-e_{[n]+i}-e_{[n]+j}}(u) \\
(I,x_{e_i}(u)) &\begin{array}{@{}r@{}l@{}} \big(\prod_{k=0}^{n-1}& x_{e_{[k]+i}-e_{[k]+(m+1)}}(\sqrt{2}u) x_{e_{[k]+(m+1)}-e_{[k]+\overline{i}}}(-\sqrt{2}u) \\
&\cdot x_{e_{[k]+i}-e_{[k]+\overline{i}}}(u^2)\big) \cdot x_{e_{[n]+i}}(u) \end{array}\\
(I,x_{-e_i}(u)) & \begin{array}{@{}r@{}l@{}}\big(\prod_{k=0}^{n-1}& x_{-e_{[k]+(m+1)}+e_{[k]+\overline{i}}}(-\sqrt{2}u) x_{-e_{[k]+i}+e_{[k]+(m+1)}}(\sqrt{2}u) \\
&\cdot x_{-e_{[k]+i}+e_{[k]+\overline{i}}}(u^2)\big)\cdot x_{-e_{[n]+i}}(u) \end{array} \\
\end{array}
\]
\end{app}


\begin{thebibliography}{99}

\bibitem[Ba17]{Baek}
S.~Baek, 
\emph{Chow Groups of Products of Severi-Brauer Varieties and Invariants of Degree 3}, 
Trans. Amer. Math. Soc. 369 (2017), no. 3, 1757--1771.

\bibitem[Bou]{Bou4-6}
N.~Bourbaki, 
\emph{\'El\'ements de Math\'ematique, Groupes et Alg\`ebres de Lie, Chapitres 4 \`a 6}, 
Masson, Paris, 1981.

\bibitem[BR13]{BR}
H.~Bermudez, A.~Ruozzi,
\emph{Degree 3 Cohomological Invariants of Groups that are Neither Simply Connected nor Adjoint},
Journal of the Ramanujan Mathematical Society 29 (2013).

\bibitem[BT72]{BT}
A.~Borel, J.~Tits,
\emph{Compl\'ements \`a l'article ``Groupes r\'eductifs"},
Publications math\'ematiques de l'I.H.\'E.S., tome 41 (1972), p. 253-276.

\bibitem[Ga09]{Gar}
S.~Garibaldi, 
\emph{Orthogonal Involutions on Algebras of Degree 16 and the Killing Form of $E_8$}, 
Contemporary Mathematics 493, 2009.

\bibitem[GMS]{GMS}
S.~Garibaldi, A.~Merkurjev, J.-P. Serre, 
\emph{Cohomological Invariants in Galois Cohomology},  University Lecture Series 28, AMS, Providence, RI, 2003.

\bibitem[GQ08]{GQ-M}
S.~Garibaldi, A.~Qu\'eguiner-Mathieu, 
\emph{Restricting the Rost Invariant to the Center}, St. Petersburg Math J. 19 (2008), no. 2, 197--213.

\bibitem[Me16]{Mer}
A.~Merkurjev. \emph{Degree Three Cohomological Invariants of Semisimple Groups}. 
J. European Math. Soc. 18 (2016), no. 2, 657--680.

\bibitem[MNZ]{MNZ}
A.~Merkurjez, A.~Neshitov, K.~Zainoulline, 
\emph{Invariants of Degree 3 and Torsion in the Chow Group of a Versal Flag}, 
Compositio Mathematica 151 (2015), 1416--1432.

\bibitem[St67]{Stein}
R.~Steinberg, 
\emph{Lectures on Chevalley Groups}, 
Yale University, 1967.

\end{thebibliography}
\end{document}